\documentclass[twoside, 14pt]{article}
\usepackage{amsmath,amssymb,amsthm,mathrsfs}
\usepackage[margin=2.2cm]{geometry} 
\usepackage{lipsum}
\usepackage{titlesec,hyperref}
\usepackage{fancyhdr}
\usepackage[numbers,sort&compress]{natbib}
\usepackage{color}
\usepackage{appendix}

\fancypagestyle{plain}
{
\fancyhf{}

}

\titleformat{\subsection}{\it}{\thesubsection.\enspace}{1.5pt}{}
\titleformat{\subsubsection}{\it}{\thesubsubsection.\enspace}{1.5pt}{}

\newtheorem{theo}{Theorem}[section]
\newtheorem{lemm}[theo]{Lemma}

\newtheorem{prop}[theo]{Proposition}
\newtheorem{rema}{Remark}[section]
\numberwithin{equation}{section}

\allowdisplaybreaks

\def\G{G_\delta}

\def\lm{\le}

\def\a{\alpha}

\def\th2{\frac{\theta}{2}}
\def\dive{\mathop{\rm div}\nolimits}

\def\R{\mathbb{R}}
\def\G{\Lambda^{-s}}
\newcommand{\bm}[1]{\mbox{\boldmath{$#1$}}}

\begin{document}
\title{Decay of Strong Solution for the Compressible Navier-Stokes Equations
        with Large Initial Data  \hspace{-4mm}}
\author{Jincheng Gao$^\dag$ \quad Zhengzhen Wei$^\ddag$ \quad Zheng-an Yao$^\sharp$ \\[10pt]
\small {School of Mathematics, Sun Yat-Sen University,}\\
\small {510275, Guangzhou, P. R. China}\\[5pt]
}

\footnotetext{Email: \it $^\dag$gaojc1998@163.com,
\it $^\ddag$weizhzh5@mail2.sysu.edu.cn,
\it $^\sharp$mcsyao@mail.sysu.edu.cn}
\date{}

\maketitle

\begin{abstract}
In this paper, we investigate the convergence of the global large solution to
its associated constant equilibrium state with an explicit decay rate
for the compressible Navier-Stokes equations in three-dimensional whole space.
Suppose the initial data belongs to some negative Sobolev space instead of Lebesgue space, we not only
prove the negative Sobolev norms of the solution being preserved along time evolution,
but also obtain the convergence of the global large solution to
its associated constant equilibrium state with algebra decay rate.
Besides, we shall show that the decay rate of the first order spatial derivative of
large solution of the full compressible Navier-Stokes equations converging to zero in
$L^2-$norm is $(1+t)^{-5/4}$, which coincides with the heat equation.
This extends the previous decay rate $(1+t)^{-3/4}$ obtained in \cite{he-huang-wang2}.
\vspace*{5pt}


\vspace*{5pt}
\end{abstract}


\section{Introduction}
\quad
   In this paper, we hope to investigate the decay estimate of strong solution
   to the isentropic compressible Navier-Stokes equations and the full compressible
    Navier-Stokes equations with large initial data in the three-dimensional whole space.
    The isentropic compressible Navier-Stokes (ICNS) equations are given as:
  \begin{equation*}\tag{ICNS}\label{cns}
  	\left\{\begin{aligned}
  &\partial_t\rho+\dive(\rho u)=0,\\
   &\partial_t(\rho u)+\dive(\rho u\otimes u)-\mu\Delta u-(\mu+\lambda)\nabla\dive u+\nabla P=0,\\
 & \underset{|x|\rightarrow\infty}{\lim}(\rho-1, u)(x, t)=(0,0),
  	\end{aligned}\right.
  \end{equation*}
where $(x, t)\in \mathbb{R}^3 \times \mathbb{R}^+$.
The unknown functions $\rho, u=(u_1, u_2, u_3)$ and $P$ represent
the density, velocity and pressure, respectively.
The pressure $P$ is given by a smooth function $P=P(\rho)=A\rho^\gamma$ with the adiabatic exponent $\gamma \ge 1$, and without loss of generality, we assume $A=1$ in this paper.
The constants $\mu$ and $\lambda$ are the viscosity coefficients,
which satisfy the following conditions:
$\mu>0$, $2\mu+3\lambda \ge0$.
To complete system \eqref{cns}, the initial data is given by
\begin{equation*}
\left.(\rho, u)(x, t)\right|_{t=0}=(\rho_0(x), u_0(x)).
\end{equation*}

There are many interesting works on the isentropic compressible Navier-Stokes equations since its physical importance, let us review some results related to well-posedness theory. In the absence of vacuum, the local well-posedness for the compressible Navier-Stokes equations  was proved by Nash \cite{nash}.
In the presence of vacuum, Huang, Li and Xin \cite{huang-li-xin} established the global existence and uniqueness of strong solution to system \eqref{cns} in three-dimensional space in the condition that the initial energy is small.
Recently, Li and Xin obtained similar results for the dimension two in \cite{li-xin}, in addition, they established the large time behavior of the solutions for system \eqref{cns} with small initial data but allowing large oscillations in \cite{li-xin}.

    The large time behavior of the solutions to the isentropic compressible Navier-Stokes system has been studied extensively. Matsumura and Nishida \cite{nishida2} first obtained the optimal decay rate for strong solution to the compressible Navier-Stokes system, and Ponce \cite{ponce} obtained the optimal $L^p(p\ge 2)$ decay rate. For more information on the long time behavior of the isentropic compressible Navier-Stokes equations with or without external forces, we refer to \cite{duan1,li-zhang}  and the references therein.  All these decay results mentioned above are restricted to the perturbation framework, that is, if the initial data is a small perturbation of constant equilibrium in $L^1\cap H^3$, then $L^2$ decay of the solution to \eqref{cns} is
$$\|\rho(t)-1\|_{L^2}+\|u (t)\|_{L^2}\le C(1+t)^{-\frac34}.$$
In the perturbation setting, the approach on proving the decay for the solutions of
compressible Navier-Stokes system rely heavily on the analysis of the linearization of the system. More precisely, most of these decay results proved by combining the linear optimal decay of spectral analysis with the energy method.

From another point of view, under the assumption that the initial perturbation is bounded in $\dot H^{-s}$, Guo and Wang \cite{guo-wang} obtained the optimal decay rates of the higher-order spatial derivatives of the solution for system \eqref{cns} by using pure energy method. Specifically, they combined energy estimates with the interpolation between negative and positive Sobolev norms and obtained the time decay rates for the isentropic compressible Navier-Stokes equations and Boltzmann equation. The new method developed in \cite{guo-wang} has a wide range of applications recently, see \cite{tan-tong,tan-wang,tan-wang-tong, tan-wang-wang, tan-wu-zhou,tan-zhou,wang}. It should be noticed that all of these decay results are established under the assumption that the initial data is a small perturbation of constant equilibrium state. As far as we know, there are no decay results established by using the method coming from \cite{guo-wang} beyond the near-equilibrium regime. Recently, He, Huang and Wang \cite{he-huang-wang} established the global stability of large solution to \eqref{cns} in $\mathbb R^3$, more precisely, they obtained the decay rate
\begin{equation*}
\|\rho(t) -1\|_{H^1}+\|u(t)\|_{H^1}\le C(1+t)^{-\frac34(\frac2p-1)},
\end{equation*}
where $(\rho_0-1,u_0)\in L^p(\mathbb R^3)\cap H^2(\mathbb R^3)$ with $p\in[1,2)$. Later, we improved this decay result in \cite{gao-wei-yao}. More precisely, we not only shown that the upper decay rate of the first order spatial derivative converging to zero in $H^1-$ norm was $(1+t)^{-\frac32(\frac1p-\frac12)-\frac12}$, but also gave the lower bound of decay rate for the global solution of \eqref{cns}  converging to constant equilibrium state in $L^2-$ norm was $(1+t)^{-\frac34}$ in the case $p=1$.
\textit{The first purpose in this paper is to investigate the convergence rate of the global large solution studied in \cite{he-huang-wang} when the initial data belongs to some negative Sobolev space $\dot H^{-s}$ instead of Lebesgue space $L^1$ .}
The advantages are that the negative Sobolev norms of the solution are shown to be preserved along time evolution
and the first order derivative of the solution can obtain the optimal convergence rate.

Second, we introduce the full compressible Navier-Stokes equations, which govern the motion of the compressible viscous and heat-conductive gases. The full compressible Navier-Stokes equations are written as follows

 \begin{equation*}\tag{FCNS}\label{fullcns0}
  	\left\{\begin{aligned}
 &\partial_t\rho+\dive (\rho u) =0,\\
  &\partial_t(\rho u)+\dive (\rho u\otimes u)+\nabla P(\rho T)=\dive \mathbb S(u),\\
   &\partial_t(\rho E)+\dive(\rho Eu)+\dive(P(\rho,T)u)=\dive(\mathbb S(u)u)+\Delta T,\\
 & \underset{|x|\rightarrow\infty}{\lim}(\rho-1, u,T-1)=0,
  	\end{aligned}\right.
  \end{equation*}
  where $\rho$, $u$, $T$ represent the density, the velocity and the absolute temperature of the fluid, respectively. In this paper, we are concerned with the perfect heat conducting and viscous gases, that is, the pressure $P(\rho,T)$ is given by $P=\rho T$. $E$ represents total energy, given by $E=T+\frac{|u|^2}{2}$. $\mathbb S(u)$ is the stress tensor:
  $$\mathbb S(u)= \mu(\nabla u+(\nabla u)^{'})+\lambda\dive u \ \mathbb I_{3\times3}.$$
  Here $\mu$ and $\lambda$ are the coefficients of viscosity, which are assumed to be constants, and satisfying $\mu>0$, $2\mu+3\lambda\ge0$.

 There are extensively literatures on system \eqref{fullcns0}. In the absence of vacuum, the global well-posedness was first established by Matsumura and Nishida \cite{nishida1} in the condition that the initial data was close to constant equilibrium in $H^3$-framework.  Matsumura and Nishida \cite{nishida2} first established the $L^2$ time-decay rate of classical solutions for system \eqref{fullcns0}. For an exterior domain in $\mathbb R^3$, Kobayashi and Shibata \cite{kobayashi2} investigated the time decay rate of solutions for \eqref{fullcns0}(see \cite{kobayashi1} for further developments). For the half space in $\mathbb R^3$, the authors investigated the asymptotic behaviors of solutions for the compressible Navier-Stokes equations in \cite{kagei1} and \cite{kagei2}.
 For the full compressible Navier-Stokes equations with external force, the authors \cite{duan22}
 obtained the convergence rate in various norms for the
 solution to the  non-trivial stationary profile in the whole space, when the initial perturbation of
 the  non-trivial stationary solution and the potential force were small in some Sobolev norms.

 In the presence of vacuum, Li and his collaborator established the global well-posedness  of strong solution for system \eqref{fullcns0} in a series of papers \cite{lijinkai1,lijinkai2,lijinkai3,lijinkai4} for the one dimensional case.  In the three dimensional case, Huang and Li \cite{huang-li} established global weak solution when the initial energy is small.
 Wen and Zhu \cite{wen-zhu} proved that the strong solution exists globally in time if the initial mass is small for the fixed coefficients of viscosity and heat conduction.
  Many other global well-posedness and the large time behavior refer to
  \cite{huang-li, danchin,qin, kawashima, jiang2} and the references therein.

 However, when the initial data is far away from the equilibrium, there are few results on the well-posedness and large-time behavior. In one-dimensional bounded domains, Kazhikhov and Shelukhin \cite{kazhikhov} first obtained the global existence of solutions to system \eqref{fullcns0} ; Jiang \cite{jiang}, Li and Liang \cite{li-liang} studied long-time behavior of large solution to system \eqref{fullcns0}. Very recently, He, Huang and Wang \cite{he-huang-wang2} proved the global-in-time stability of large solution for system \eqref{fullcns0} under the assumption that
 the density is uniformly bounded in $C^\alpha$ for some small $0<\alpha<1$ in the whole space. More precisely, they got the uniform-in-time bound for the global solution by using some techniques about the blow-up criterion come from \cite{sun-wang, wen-zhu2, lijing2}, and they obtained decay rate
 \begin{equation}\label{00}
 \|u(t)\|_{H^1}+\|\rho(t)-1\|_{H^1}+\|T(t)-1\|_{H^1}\le C(1+t)^{-\frac34}.
 \end{equation}
Here the initial data $(\rho_0-1,u_0)\in L^1(\mathbb R^3)\cap H^2(\mathbb R^3)$ and $T_0-1\in L^1(\mathbb R^3)\cap H^1(\mathbb R^3)$. It is obviously that the decay rate of the first order spatial derivative of solution in \eqref{00} is not optimal.
Thus the second aim of this paper is to address the following three problems
for the global solution investigated in \cite{he-huang-wang2}:
\textit{(i)obtain optimal decay rate(at least faster than $(1+t)^{-\frac{3}{4}}$) for the first order
 spatial derivative of global solution;
(ii)prove that the second order spatial derivative of global solution will converge to zero;
(iii)establish the decay rate for the global solution when the initial perturbation
   belongs to some negative Sobolev space $\dot H^{-s}$ instead of Lebesgue space $L^1$.}

\textbf{Notation:} In this paper, we use $H^s(s\in\mathbb R)$ to denote the usual Sobolev space with norm $\|\cdot\|_{H^s}$ and $L^p(\mathbb{R}^3)$ to denote the usual $L^p$ space with norm $\|\cdot\|_{L^p}$. $\mathcal{F}(f):=\widehat{f}$ represents the usual Fourier transform of the function $f$.
The operator $\Lambda^s, s\in\mathbb R$ defined as  $$\Lambda^s f(x)=\int_{\mathbb R^3}|\xi|^s\hat f(\xi)e^{2\pi ix \cdot \xi}d \xi.$$
Then the definition of the Sobolev space is given as
 $$\|f\|_{\dot H^s}:=\|\Lambda^s f\|_{L^2}=\||\xi|^s\hat f\|_{L^2}.$$
 When the index $s$ is non-positive, we called the space as negative Sobolev space. For convenience, we will change the index to be $``-s"$ with $s\ge0$. For the sake of simplicity, we write $\int f dx:=\int _{\mathbb{R}^3} f dx$
 and $\|(A, B)\|_X:=\| A \|_X+\| B\|_X$. We will use the notation $a\sim  b$ whenever $a\le C_1b$ and $b\le C_2a$ where $C_1$ and $C_2$ are universal constants. We will employ the notation $a\lesssim b$ to mean that $a\le Cb$ for a universal constant $C>0$ that only depends on the parameters coming from the problem but independent of time, and may change from line to line.

This paper is organized as follows. In Section \ref{2}, we will state the main results of this paper and comment the analysis of the proofs. In Section \ref{3}, we will prove the main theorem of the isentropic compressible Navier-Stokes equations. In Section \ref{4}, we will prove the main two theorems of the full compressible Navier-Stokes equations. In Appendix \ref{appendix}, we present some useful inequalities, which are extensively used in this paper.

\section{Main results}\label{2}
\subsection{Main results for the isentropic compressible Navier-Stokes equations.}
\quad
In this subsection, we are concerned with the decay rate
for large solution of the system \eqref{cns}. First of all, we recall the following results obtained in \cite{he-huang-wang}, which will be used in this paper frequently.
 \begin{theo}(see \cite{he-huang-wang})\label{Thm1}
 Let $\mu>\frac12\lambda$, and $(\rho,u)$ be a global and smooth solution of \eqref{cns} with initial data $(\rho_0,u_0)$ where $\rho_0\ge c>0$. Suppose the admissible condition holds:
 $$ \partial_tu\mid_{t=0}=-u_0\cdot\nabla u_0+\frac{1}{\rho_0}Lu_0-\frac{1}{\rho_0}\nabla \rho_0^\gamma,$$
 where operator $L$ is defined by $Lu=-\dive(\mu\nabla u)-\nabla((\lambda+\mu)\dive u)$. Assume that $a:=\rho-1$, and $\sup_{t\ge0}\|\rho(t)\|_{C^\a}\lm M$ for small $0<\a<1$. Then if $a_0,u_0 \in L^{p}(\R^3)\cap H^2(\R^3)$ with $p\in[1,2]$, we have\\
(1)\textbf{(Lower bound of the density)}\\
 \quad There exists a positive constant $\underline{\rho}=\underline{\rho}(c,M)$
 such that for all $t\ge0$
 \begin{equation}\label{Lower}
 \rho(t)\ge\underline{\rho}.
  \end{equation}
(2)\textbf{(Uniform-in-time bounds for the regularity of the solution)}
 \begin{equation}\label{uniform-bound}
 \|a\|^2_{L^\infty(H^2)}+\|u\|^2_{L^\infty(H^2)}+\int_0^\infty(\|\nabla a\|^2_{H^1}+\|\nabla u\|^2_{H^2})d\tau\lm C(\underline{\rho},M,\|a_0\|_{H^2},\|u_0\|_{H^2}).
 \end{equation}
(3)\textbf{(Decay estimate for the solution)}
 \begin{equation}\label{decay}
 \|u(t)\|_{H^1}+\|a(t)\|_{H^1}\lm
 C(M,\|a_0\|_{L^{p}\cap H^1},\|u_0\|_{L^{p}\cap H^2})(1+t)^{-\beta(p)},
 \end{equation}
  where $\beta(p)=\frac34(\frac{2}{p}-1)$.
 \end{theo}

Then our first main result is stated in the following theorem:

 \begin{theo}\label{Thm2}
 Suppose all the conditions in Theorem \ref{Thm1} hold on, except for $a_0:=\rho_0-1\in \dot H^{-s}$, and $u_0\in \dot H^{-s}$ with $s\in(0,3/2)$. Let $(\rho, u)$ be the global solution of \eqref{cns},  then it hold on for all $t\ge 0$,
 \begin{equation*}
 \|\G a(t)\|^2_{L^2}+\|\G u(t)\|^2_{L^2}+\|\G\dot u(t)\|_{L^2}^2\le C,
 \end{equation*}
 and
 \begin{equation*}
 \begin{aligned}
 \|a(t)\|^2_{H^1}+\|u(t)\|^2_{H^1}+\|\dot u(t)\|_{L^2}^2\le C(1+t)^{-s}.
 \end{aligned}
 \end{equation*}
 For all $t\ge T_1$, there holds
 \begin{equation*}
 \|\nabla a(t)\|_{H^1}^2+\|\nabla u(t)\|_{H^1}^2\le C(1+t)^{-(1+s)},
 \end{equation*}
 where $T_1$ is constant given in Lemma \ref{diss}, and $C$ is a constant independent of time.
 \end{theo}

\begin{rema}
The reason for the constraint $s<3/2$ comes from applying Lemma \ref{l2} to estimate the nonlinear terms when doing the negative Sobolev estimate via $\G$. For $s\ge 3/2$, the nonlinear estimates would not work.
\end{rema}

\begin{rema}
Compared with the usual $L^p-L^2(p\in[1, 2))$ approach of investing the optimal decay of the solution, an important feature is that the $\dot H^{-s}$ norm of the solution is preserved along time evolution, but it is difficult to prove that the $L^p$ norm of the solution can be preserved along time evolution. It should be noticed that both $\dot H^{-s}$ and $L^p$ norms enhance the decay rate of the solution. When $s\in (0,3/2)$, there is the embedding relation
$\dot H^{-s}\hookrightarrow L^p, \quad p\in(1,2).$
\end{rema}

The proof of Theorem \ref{Thm2} will be given in Section \ref{3}, we now make some comments on the analysis of the proof. The main idea of the proof is combining the energy estimates with the interpolation between negative and positive Sobolev norms as Lemma \ref{l1}. Thus it is important to prove the $\dot H^{-s}$ norm of the solution is preserved along time evolution. When deriving the negative Sobolev estimates, we use different interpolation inequalities to estimate the nonlinear terms in the case $s\in(0,1/2]$ and $s\in(1/2,3/2)$ respectively. Thus we need to split the proof of Theorem \ref{Thm2} into two parts. When $s\in(0,1/2]$, we can check that the quantity $(\G a, \G(\rho u))$ is uniform bounded respect to time
since $\int_0^t(\|\nabla u\|_{H^1}^2+\|\nabla a\|_{H^1}^2)d\tau$ is uniform bounded respect to time.
Therefore, we can establish the decay estimate
$$\|a(t)\|_{H^1}^2+\|u(t)\|_{H^1}^2+\|\dot u(t)\|_{L^2}^2\le C (1+t)^{-s}, \quad s\in(0,1/2].$$
When $s\in(1/2,3/2)$, in order to derive the propagation of the negative Sobolev norms of the solution,
it is necessary to check that
$\int_0^t\|(a,u)\|_{L^2}^{s-1/2}\|\nabla(a,u)\|_{L^2}^{5/2-s}d\tau$
is uniform bounded respect to time. To achieve that, we need to improve the decay rate for the first order derivative of solution when $s\in(0,1/2]$. According to Lemma \ref{diss}, it is easy to derive that there exists a $T_1>0$, there holds
$$\|\nabla a (t)\|_{L^2}^2+\|\nabla u(t)\|_{L^2}^2
\le C(1+t)^{-(1+s)}, \quad t\ge T_1.$$
Once this decay rate are obtained, the propagation of the negative Sobolev norms of the solution in the case $s\in(1/2,3/2)$ follows by the estimates in Section \ref{3} and the uniform bound \eqref{uniform-bound}.

\subsection{Main results for the full compressible Navier-Stokes equations.}
\quad In this subsection, we state the main results of the full compressible Navier-Stokes equations \eqref{fullcns0}.
   Before  stating our main results, we need to introduce the following results obtained in \cite{he-huang-wang2}, which will be used frequently.

   \begin{theo}(see \cite{he-huang-wang2})\label{0them}
   Let $\mu>\frac12\lambda$, and $(\rho,u,T)$ be a global and smooth solution of \eqref{fullcns0} with initial data $(\rho_0,u_0,T_0)$ where $\rho_0\ge c>0$, $T_0\ge c>0$, and satisfying the admissible conditions:
   \begin{equation}\label{admissible}
   \begin{aligned}
  & u_t|_{t=0}=-u_0\cdot\nabla u_0+\frac{1}{\rho_0}\dive(\mathbb S(u_0))-\frac{1}{\rho_0}\nabla(\rho_0 T_0),\\&
   T_t|_{t=0}=-u_0\cdot \nabla T_0-T_0\dive u_0+\frac{1}{\rho_0}\mathbb S(u_0):\nabla u_0+\frac{1}{\rho_0}\Delta T_0.
    \end{aligned}
   \end{equation}
   Assume that $(\rho, T)$ satisfies that
   \begin{equation}\label{thetarho}
   \|\rho,T\|_{L^{\infty}}\le M_1,\quad \|\rho\|_{C^{\alpha}}\le M_2,
   \end{equation}
   where $\alpha$ is a positive and sufficiently small constant. Denote that $a\overset{def}{=}\rho-1$, $\theta\overset{def}{=}T-1$, then if $(a_0,u_0)\in L^1(\mathbb R^3)\cap H^2(\mathbb R^3)$ and $\theta_0\in L^1(\mathbb R^3)\cap H^1(\mathbb R^3)$, we have \\
(1)   \textbf{(Propagation of the lower bounds of the density and the temperature)}\\
   \quad There exist two constants $c_1$ and $c_2$ depending on $\alpha$ and $M$ such that
   $$\rho(t,x)\ge c_1>0,\quad T(t,x)\ge c_2>0.$$
  (2)\textbf{(Uniform-in-time bounds for the regularity)}
  \begin{equation}\label{uniform}
  \|a,u\|^2_{L^\infty_tH^2}+\|\theta\|^2_{L^\infty_tH^1}+\int_0^\infty(\|(a,\nabla u)\|_{H^2}^2+\|\nabla\theta\|^2_{H^1})d\tau\le C(M,\|a_0\|_{L^1\cap H^2},\|u_0\|_{L^1\cap H^2},\|\theta_0\|_{L^1\cap H^1}),
  \end{equation}
 (3)\textbf{(Long time behavior of the solution)}\\
 \begin{equation}\label{decay22}
 \|u(t)\|_{H^1}+\|a(t)\|_{H^1}+\|\theta(t)\|_{H^1}\le C(M,\|a_0\|_{L^1\cap H^2},\|u_0\|_{L^1\cap H^2},\|\theta_0\|_{L^1\cap H^1})(1+t)^{-\frac34},
 \end{equation}
 where $M=\max\{M_1,M_2\}$.
   \end{theo}

Now we are in the position to state our main results for the full compressible Navier-Stokes equations.
\begin{theo}\label{Thm3}
Under the assumptions of Theorem \ref{0them}, assume $(\rho, u,\theta)$ be the global solution of system \eqref{fullcns0}, then we have
\begin{equation}\label{Thm3decay}
\|\nabla a\|_{H^1}^2+\|\nabla u\|_{H^1}^2+\|\nabla\theta\|_{H^1}^2+\|\partial_ta\|_{L^2}^2+\|\partial_tu\|_{L^2}^2+\|\partial_t\theta\|_{L^2}^2\le C(1+t)^{-\frac52}, \quad t\ge T_3,
\end{equation}
where $T_3$ is a large constant given in Lemma \ref{lemma4}, and $C$ is a constant independent of time.
\end{theo}

 \begin{rema}\label{nabla2theta}
 Although the initial data $\theta_0\in H^1(\mathbb R^3)$ given in Theorem \ref{Thm3},
 it is easy to check that $\theta$ actually belongs to $H^2(\mathbb R^3)$ duing to the
 compatibility conditions \eqref{admissible}. In other words, we have
   \begin{equation}\label{thetabound}
  \|\nabla^2\theta\|_{L^\infty(0,\infty;L^2)}<\infty.
  \end{equation}
   This proof can be found in Corollary 2.1 in \cite{he-huang-wang2}.
  \end{rema}

\begin{rema}
Compared with the decay rate \eqref{decay22}, our decay results not only imply that the second order spatial derivative of solution converges to zero, but also shows that the decay rate for the first order spatial derivative of solution is optimal in the sense that it coincides with the decay rate of the solution to the heat equation.
\end{rema}

\begin{rema}
By the Sobolev interpolation inequality, it is shown that the solution $(\rho,u, T)$ converges to the constant equilibrium state $(1,0,1)$ at the $L^q(2\le q\le 6)$-rate $(1+t)^{-\frac34-\frac{3q-6}{4q}}$.
\end{rema}

\begin{rema}
In our previous work \cite{gao-wei-yao}, we obtained the similar decay results on the isentropic compressible Navier-Stokes equations. In above theorem, Since the pressure $P=\rho\theta$ in the system \eqref{fullcns0}, and we only have the uniform bound $\|\theta\|_{L^\infty_t(H^1)}\le C$ in \eqref{uniform}, thus the energy estimate for the second order spatial derivative of solution is more complicated.
\end{rema}

Let us give some comments on the analysis for the proof of Theorem \ref{Thm3}.
Since the solution itself and its first order spatial derivative decay with the same $L^2-$ rate $(1+t)^{-\frac34}$,
these quantities can be small enough if the time is large.
Thus, we take the strategy of the frame of small initial data (cf.\cite{nishida1}) to establish the energy estimate:
\begin{equation}\label{ana}
\begin{aligned}
\frac{d}{dt}&\mathcal E_2^2(t)+c_0(\|\nabla^2 u\|_{H^1}^2+\|\nabla^2\theta\|_{H^1}^2+\|\nabla^2a\|_{L^2}^2)\\&
\lesssim Q(t)(\|\nabla^2u\|_{H^1}^2+\|\nabla^2\theta\|_{H^1}^2+\|\nabla^2a\|_{L^2}^2)
+\|\nabla u\|_{L^\infty}\|\nabla^2a\|_{L^2}^2+\|\theta\|_{L^\infty}\|\nabla^2a\|_{L^2}\|\nabla^3u\|_{L^2},
\end{aligned}
\end{equation}
where $\mathcal E_2^2(t)$ is equivalent to $\|\nabla(a,u,\theta)\|_{H^1}^2$,
and $Q(t)$ consists of some difficult terms, such as $\|(a,u)\|_{L^\infty}$ and $\|\nabla(a,u)\|_{L^3}$.
To close the estimate, the difficulty is that we can only use the smallness of quantities
in $H^1$ norm  rather than $H^2$ norm. 
Thus, our idea is to apply the Sobolev interpolation inequality to control these quantities by the product
of solution itself and the second order spatial derivative. Since
the latter one is uniform bounded with respect to time(see \eqref{uniform}), $Q(t)$ is a
small quantity which appears as a prefactor in front of dissipation term. 
And hence, $Q(t)$ is a small quantity actually after some large time.
The second and third terms on the righ-thand side of inequality \eqref{ana} can be controlled by the similar
method, see \eqref{df} and \eqref{df2} in detail.
Then, the terms on the right-hand side of \eqref{ana} can be absorbed into the second term
on the left-hand side of inequality \eqref{df} after a fixed large time.
Finally,  we hope to perform the upper decay rate \eqref{Thm3decay} by using
the energy inequality \eqref{ana} and the time-frequency splitting method by Schonbek \cite{schonbek1}.
However, unlike the incompressible flow(cf.\cite{schonbek2, Schonbek-Wiegner}), the dissipation of density is weaker than the one of the velocity and temperature for the full compressible Navier-Stokes equations.
To obtain the dissipative estimate of density, we will weaken the coefficients of velocity and temperature dissipation, and thus one part of the dissipation of density will play a role of damping term. Finally, according to time-frequency splitting method, we can derive decay estimate \eqref{Thm3decay}.

Our third result investigates the decay rate for the global solution of system \eqref{fullcns0} in the case that the initial data $(a_0, u_0,\theta_0)\in \dot H^{-s}$ with $s\in(0,3/2)$, stated as follows
 \begin{theo}\label{1theo}
 Under the assumptions of Theorem \ref{0them}, but $(a_0, u_0,\theta_0)\in \dot H^{-s}$ with $s\in(0,3/2)$. Let $(\rho, u,\theta)$ be the global solution of system \eqref{fullcns0},  then for all $t\ge 0$,
 \begin{equation*}
 \|\G a(t)\|^2_{L^2}+\|\G u(t)\|^2_{L^2}+\|\G\dot u(t)\|_{L^2}^2+\|\G\theta(t)\|_{L^2}^2\le C,
 \end{equation*}
 and
 \begin{equation*}
 \begin{aligned}
 \|a(t)\|^2_{H^1}+\|u(t)\|^2_{H^1}+\|\theta(t)\|_{H^1}^2+\|\dot u(t)\|_{L^2}^2\le C(1+t)^{-s}.
 \end{aligned}
 \end{equation*}
 For $t\ge T_4$, there holds
\begin{equation*}
\|\nabla a(t)\|_{H^1}^2+\|\nabla u(t)\|_{H^1}^2+\|\nabla\theta(t)\|_{H^1}^2\le C(1+t)^{-(1+s)}, \end{equation*}
where $T_4$ is given in Proposition \ref{prop4}, $C$ is a constant independent of time.
 \end{theo}

The proof of Theorem \ref{1theo} is similar as the proof of Theorem \ref{Thm2} in Section \ref{3}, except that additional estimates of the temperature equation are established. And it should be noticed that we only have the uniform bound $\|\theta\|_{L^\infty_t(H^1)}\le C$ in \eqref{uniform}, thus the energy estimate for the second order spatial derivative of solution is more complicated, which is necessary for improving the decay rate for the first order derivative of solution when $s\in(0,1/2]$.

\section{The proof of the results for ICNS}\label{3}

\quad First, we need to derive the evolution of the negative Sobolev norms of the solution for the isentropic compressible Navier-Stokes equations \eqref{cns}.

\textbf{Energy evolution of negative Sobolev norms.} In what follows, we will derive the evolution of the negative Sobolev norms of the solution of system \eqref{cns}. In order to estimate the nonlinear terms, we need to restrict ourselves to that $s\in(0,3/2)$. Recalling $a:=\rho-1$, we rewrite \eqref{cns} as
\begin{equation}\label{cns1}
  	\left\{\begin{aligned}
  &\partial_ta+\dive u=-\dive(au),\\
   &\partial_t(\rho u)+\dive(\rho u\otimes u)-\mu\Delta u-(\mu+\lambda)\nabla\dive u+\gamma\nabla a+(\gamma-1)\nabla H(\rho|1)=0,\\
  	\end{aligned}\right.
  \end{equation}
  where
\begin{equation}
H(\rho|1)=\left\{
\begin{array}{cl}
&\frac{1}{\gamma-1}(\rho^{\gamma}-1-\gamma(\rho-1)), \quad {\rm when} \quad \gamma>1,\\
&\rho\ln\rho-\rho+1,   \qquad {\rm when} \quad \gamma=1.\\
\end{array}\right.
\end{equation}

\begin{lemm}\label{l3}
 For $s\in(0, 1/2]$, we have
 \begin{equation}\label{est11}
 \begin{aligned}
 \frac12&\frac{d}{dt}\int(\gamma|\G a|^2+|\G(\rho u)|^2)dx+C\int|\nabla\G u|^2dx\\&
 \lesssim (\|\nabla u\|_{H^1}^2+\|\nabla a\|^2_{H^1})(\|\G(\rho u)\|_{L^2}+\|\G a\|_{L^2}+\|\nabla\G u\|_{L^2}) +\|\nabla u\|^2_{H^1}\|\nabla \rho\|_{H^1},
  \end{aligned}
 \end{equation}
 and for $s\in(1/2,3/2)$, we have

 \begin{equation}\label{est12}
\begin{aligned}
\frac12&\frac{d}{dt}\int(\gamma|\G a|^2+|\G(\rho u)|^2)dx+C\int|\nabla\G u|^2dx\\&
\lesssim \|(a,u)\|_{L^2}^{s-1/2}\|(\nabla a, \nabla u)\|_{L^2}^{5/2-s}(\|\G a\|_{L^2}+\|\G(\rho u)\|_{L^2})\\&
\quad +\|(a,u)\|_{L^2}^{s-1/2}\|(\nabla a, \nabla u)\|_{L^2}^{3/2-s}\|(a, u)\|_{H^1}\|\nabla\G u\|_{L^2}\\&
\quad +\|u\|_{L^2}^{2s}\|\nabla u\|_{L^2}^{3-2s}\|\nabla a\|_{L^2}+\|a\|_{L^2}^{s-1/2}\|\nabla a\|_{L^2}^{3/2-s}\|u\|_{L^2}^{s+1/2}\|\nabla u\|_{L^2}^{5/2-s}.
\end{aligned}
\end{equation}
\end{lemm}

\begin{proof}
Applying $\Lambda ^{-s}$ to $\eqref{cns1}_1$, $\eqref{cns1}_2$ and multiplying the resulting by $\gamma\Lambda^{-s}a$, $\Lambda^{-s}(\rho u)$ respectively, summing up and then integrating over $\mathbb R^3$ by parts, we obtain
\begin{equation}\label{est1}
\begin{aligned}
&\frac12\frac{d}{dt}\int(\gamma|\G a|^2+|\G(\rho u)|^2)dx+\mu\int|\nabla\G u|^2dx+(\mu+\lambda)\int|\dive \G u|^2dx\\&
=-\int\G\nabla\left((\gamma-1)H(\rho|1)\right)\cdot \G(\rho u)dx-\int\G\dive(\rho u\otimes u)\cdot \G(\rho u)dx\\&
\quad -\gamma\int \G\dive(au)\cdot \G adx+\mu\int\G(\Delta u)\cdot\G(au)dx\\&
\quad +(\mu+\lambda)\int\G(\nabla\dive u)\cdot\G(au)dx\\&
:=I_1+I_2+I_3+I_4+I_5.
\end{aligned}
\end{equation}

In order to estimate the nonlinear terms in the right-hand side of \eqref{est1}, we restrict the value of $s$. If $s\in(0,1/2]$, then $1/2+s/3<1$ and $3/s\ge 6$. Then thanks to \eqref{a.3}, we obtain
\begin{equation}\label{H0}
\begin{aligned}
I_1&=-\int\G\nabla\left((\gamma-1)H(\rho|1)\right)\cdot \G(\rho u)dx\\&
\lesssim \|\G(\nabla H(\rho|1))\|_{L^2}\|\G(\rho u)\|_{L^2}\\&
\lesssim \|\nabla H(\rho|1)\|_{L^{\frac{1}{1/2+s/3}}}\|\G(\rho u)\|_{L^2}\\&
\lesssim \|a\nabla a\|_{L^{\frac{1}{1/2+s/3}}}\|\G(\rho u)\|_{L^2}\\&
\lesssim \|a\|_{L^{3/s}}\|\nabla a\|_{L^2}\|\G(\rho u)\|_{L^2}\\&
\lesssim \|\nabla a\|_{L^2}^{1/2+s}\|\nabla^2a\|_{L^2}^{1/2-s}\|\nabla a\|_{L^2}\|\G(\rho u)\|_{L^2}\\&
\lesssim (\|\nabla a\|_{L^2}^2+\|\nabla ^2a\|_{L^2}^2)\|\G(\rho u)\|_{L^2}\\&
\lesssim \|\nabla a\|^2_{H^1}\|\G(\rho u)\|_{L^2},
\end{aligned}
\end{equation}
where we have used  $\underline \rho\le \rho\le M$
and $|\nabla H(\rho|1)|\lesssim |a\nabla a|$, which is not difficult to obtain by Taylor expansion.

Integrating by part and applying the uniform bound \eqref{uniform-bound}, we have
\begin{equation*}
\begin{aligned}
I_2&=\int \G(\rho u\otimes u)\cdot \G(\nabla\rho u+\rho\dive u)dx\\&
\lesssim \|\G(\rho u\otimes u)\|_{L^2}(\|\G(\nabla \rho u)\|_{L^2}+\|\G(\rho\dive u)\|_{L^2})\\&
\lesssim \|\rho\|_{L^\infty}\|u\|_{L^2}\|u\|_{L^{3/s}}(\|u\|_{L^{3/s}}\|\nabla\rho\|_{L^2}+\|\rho\|_{L^{3/s}}\|\nabla u\|_{L^2})\\&
\lesssim \|\nabla u\|_{L^2}^{1/2+s}\|\nabla^2 u\|^{1/2-s}_{L^2}(\|\nabla u\|_{L^2}^{1/2+s}\|\nabla^2u\|_{L^2}^{1/2-s}\|\nabla\rho\|_{L^2}+\|\nabla\rho\|_{L^2}^{1/2+s}\|\nabla^2\rho\|_{L^2}^{1/2-s}\|\nabla u\|_{L^2})\\&
\lesssim \|\nabla u\|^2_{H^1}\|\nabla\rho\|_{H^1}.
\end{aligned}
\end{equation*}
Similarly, it is easy to check that
\begin{equation*}
\begin{aligned}
I_3\lesssim (\|\nabla u\|^2_{H^1}+\|\nabla a\|^2_{H^1})\|\G a\|_{L^2}.
\end{aligned}
\end{equation*}
Integrating by part and using the same method as $I_2$, one arrives at
\begin{equation*}
\begin{aligned}
I_4&=\mu\int\G(\Delta u)\G(au)dx\\&
=-\mu \int\nabla\G u\cdot\G(\nabla au+a\nabla u)dx\\&
\lesssim \|\nabla\G u\|_{L^2}(\|\nabla u\|^2_{H^1}+\|\nabla a\|^2_{H^1}).
\end{aligned}
\end{equation*}
Similarly, it holds on
\begin{equation*}
I_5\lesssim \|\nabla\G u\|_{L^2}(\|\nabla u\|^2_{H^1}+\|\nabla a\|_{H^1}^2).
\end{equation*}
Collecting all above estimates, we deduce \eqref{est11}.

Now if $s\in (1/2,3/2)$, we shall estimate the right-hand side of \eqref{est1} in a different way. Since $s\in(1/2,3/2)$, we have $1/2+s/3<1$ and $2<3/s<6$. Then using the different Sobolev interpolation, we have
\begin{equation*}
\begin{aligned}
I_1&=\int\G(\nabla(\gamma-1)H(\rho|1))\cdot\G(\rho u)dx\\&
\lesssim \|a\|_{L^{3/s}}\|\nabla a\|_{L^2}\|\G(\rho u)\|_{L^2}\\&
\lesssim \|a\|_{L^2}^{s-1/2}\|\nabla a\|_{L^2}^{3/2-s}\|\nabla a\|_{L^2}\|\G(\rho u)\|_{L^2}.
\end{aligned}
\end{equation*}
Integrating by part, it holds on
\begin{equation*}
\begin{aligned}
I_2&=\int\G(\rho u\otimes u)\cdot\G(\nabla\rho u+a\dive u+\dive u)dx\\&
\lesssim \|\G(\rho u\otimes u)\|_{L^2}(\|\G(\nabla\rho u)\|_{L^2}+\|\G(a\dive u)\|_{L^2}+\|\G(\dive u)\|_{L^2})\\&
\lesssim \|\rho\|_{L^\infty}\|u\|_{L^{3/s}}\|u\|_{L^2}(\|u\|_{L^{3/s}}\|\nabla \rho\|_{L^2}+\|a\|_{L^{3/s}}\|\nabla u\|_{L^2}+\|\nabla\G u\|_{L^2}).
\end{aligned}
\end{equation*}
Using the Sobolev interpolation
\begin{equation*}
\|f\|_{L^{3/s}}\lesssim \|f\|_{L^2}^{s-1/2}\|\nabla f\|_{L^2}^{3/2-s}, \quad s\in(1/2,3/2)
\end{equation*}
again, it is easy to derive
\begin{equation*}
\begin{aligned}
I_2&\lesssim \|u\|_{L^2}^{2s}\|\nabla u\|_{L^2}^{3-2s}\|\nabla\rho\|_{L^2}+\|u\|_{L^2}^{s+1/2}\|\nabla u\|_{L^2}^{3/2-s}\|\nabla\G u\|_{L^2}\\&
\quad +\|a\|_{L^2}^{s-1/2}\|\nabla a\|_{L^2}^{3/2-s}\|u\|_{L^2}^{s+1/2}\|\nabla u\|_{L^2}^{5/2-s}.
\end{aligned}
\end{equation*}
Similarly, we have
\begin{equation*}
I_3\lesssim (\|u\|_{L^2}^{s-1/2}\|\nabla u\|_{L^2}^{3/2-s}\|\nabla a\|_{L^2}+\|a\|_{L^2}^{s-1/2}\|\nabla a\|_{L^2}^{3/2-s}\|\nabla u\|_{L^2})\|\G a\|_{L^2},
\end{equation*}
\begin{equation*}
I_4\lesssim \|\G\nabla u\|_{L^2}(\|u\|_{L^2}^{s-1/2}\|\nabla u\|_{L^2}^{3/2-s}\|\nabla a\|_{L^2}+\|a\|_{L^2}^{s-1/2}\|\nabla a\|_{L^2}^{3/2-s}\|\nabla u\|_{L^2}),
\end{equation*}
\begin{equation*}
I_5\lesssim \|\G\nabla u\|_{L^2}(\|u\|_{L^2}^{s-1/2}\|\nabla u\|_{L^2}^{3/2-s}\|\nabla a\|_{L^2}+\|a\|_{L^2}^{s-1/2}\|\nabla a\|_{L^2}^{3/2-s}\|\nabla u\|_{L^2}).
\end{equation*}
Collecting all above estimates together, we obtain \eqref{est12}. Thus we complete the proof of lemma.
\end{proof}

Before giving decay result when $s\in(0,1/2]$, we need to introduce the dissipation inequality obtained in \cite{he-huang-wang}.

\begin{prop}(see \cite{he-huang-wang})\label{pro1}
 Under the assumptions of Theorem \ref{Thm1}, then there holds
 \begin{equation}\label{est2}
 \frac{d}{dt}X_1(t)+C(\|\nabla^2u\|^2_{L^2}+\|\nabla u\|^2_{L^2}+\|\nabla a\|^2_{L^2}+\|\nabla\dot u\|^2_{L^2})\le 0,
 \end{equation}
 where $X_1(t)\sim \|u\|_{H^1}^2+\|a\|^2_{H^1}+\|\dot u\|^2_{L^2}$.
  \end{prop}

\begin{prop}\label{prop1}
Under the assumptions of Theorem \ref{Thm2}, when $s\in(0,1/2]$, then for all $t\ge 0$, we have
 \begin{equation}\label{prop1est}
 \begin{aligned}
 &\|\G a(t)\|_{L^2}^2+ \|\G u(t)\|_{L^2}^2+ \|\G \dot u(t)\|_{L^2}^2\le C,\\&
 \|a(t)\|_{H^1}^2+\|u(t)\|_{H^1}^2+\|\dot u(t)\|_{L^2}^2\le C(1+t)^{-s},
 \end{aligned}
 \end{equation}
where $C$ is a constant independent of time.
\end{prop}
\begin{proof}
Integrating \eqref{est11} over $[0,t)$ and using Cauchy inequality, we get
\begin{equation}\label{331}
\begin{aligned}
\frac12&\sup_{\tau\in[0,t)}(\|\G a(\tau)\|^2_{L^2}+\|\G(\rho u)(\tau)\|^2_{L^2}) +\frac C2\int_0^t\|\nabla\G u\|^2_{L^2}d\tau\\&
\le \frac12(\|\G a_0\|^2_{L^2}+\|\G(\rho_0u_0)\|^2_{L^2})\\&
\quad +\sup_{\tau\in[0,t)}(\|\G a(\tau)\|_{L^2}+\|\G(\rho u)(\tau)\|_{L^2})\int_0^t(\|\nabla u\|^2_{H^1}+\|\nabla a\|^2_{H^1})d\tau\\&
\quad +\epsilon(\|\nabla u\|_{L^\infty_\tau(H^1)}^2+\|\nabla a\|_{L^\infty_\tau(H^1)}^2)\int_0^t \|\G\nabla u\|_{L^2}^2d\tau\\&
\quad +C_\epsilon\int_0^t(\|\nabla u\|_{H^1}^2+\|\nabla a\|_{H^1}^2)d\tau+\|\nabla\rho\|_{L^\infty_\tau(H^1)}\int_0^t\|\nabla u\|_{H^1}^2d\tau.
\end{aligned}
\end{equation}
Notice that $\underline \rho\le\rho_0\le M$, and using Hausdorff-Young inequality \eqref{hausdorff}, we have
\begin{equation*}
\begin{aligned}
\|\rho_0u_0\|_{\dot H^{-s}}=\||\xi|^{-s}\widehat{\rho_0u_0}\|_{L^2}&\le \||\xi|^{-s}\hat\rho_0*\hat u_0\|_{L^2}\le\|\hat\rho_0\|_{L^1}\||\xi|^{-s}\hat u_0\|_{L^2}
\le \|\rho_0\|_{L^\infty}\|u_0\|_{\dot H^{-s}}
\le C\|u_0\|_{\dot H^{-s}},
\end{aligned}
\end{equation*}
this together with $u_0\in\dot H^{-s}$ implies $\|\G(\rho_0 u_0)\|_{L^2}\le C$.
Then according to the uniform bound \eqref{uniform-bound}
and choosing $\epsilon$ small enough in inequality \eqref{331}, we obtain
\begin{equation}\label{propa}
\|\G a(t)\|^2_{L^2}+\|\G(\rho u)\|^2_{L^2}\le C, \quad s\in(0,1/2].
\end{equation}
Since $\underline\rho\le\rho\le M$, it is easy to check $\widehat{\rho u}\sim \hat u$, hence we have
\begin{equation}\label{propu}
\|\G u(t)\|_{L^2}^2\le \|\G(\rho u)(t)\|_{L^2}^2\le C.
\end{equation}

Next, let us prove $\|\G\dot u(t)\|\le C$.
We have known $u\in \dot H^{-s}\cap H^2$, thus we have $u\in \dot H^{2-s}$ by Sobolev interpolation inequality.  Then we have
\begin{equation}\label{prop1est1}
\||\xi|^{-s}\widehat {\Delta u}\|_{L^2}\le \||\xi|^{2-s}\hat u\|_{L^2}\le \|u\|_{\dot H^{2-s}}\le C,
\end{equation}
which implies
\begin{equation}\label{lapu}
\Delta u\in \dot H^{-s}.
\end{equation}

Similarly, we have $a\in\dot H^{1-s}$.  In follows, we set ${\bm a}\overset{\Delta}{=} \rho^\gamma-1$. Observing thet ${\bm a}=(\int_0^1\gamma(t\rho+(1-t))^{\gamma-1}dt)a$.  Since $\underline \rho\le \rho\le M$, it is easy to see
 \begin{equation*}
 \bm a\le \gamma(M+1)^{\gamma-1}a,
 \end{equation*}
 which implies $\hat{\bm a}\lesssim \hat a$. Thus we have
\begin{equation*}
\||\xi|^{-s}\widehat {\nabla{\bm a}}\|_{L^2}\le \||\xi|^{1-s}\hat{\bm a}\|_{L^2}\le C \||\xi|^{1-s}\hat a\|_{L^2}\le C\|a\|_{\dot H^{1-s}}\le C,
\end{equation*}
which means
\begin{equation}\label{aa}
\nabla {\bm a}\in \dot H^{-s}.
\end{equation}
Since $\rho\dot u+\nabla {\bm a}=\mu\Delta u+(\lambda+\mu)\nabla\dive u$, combing \eqref{lapu} and \eqref{aa}, we have $\rho\dot u\in \dot H^{-s}$. Using  $\underline\rho\le\rho\le M$ again, we deduce that
\begin{equation*}
\dot u\in \dot H^{-s}.
\end{equation*}
So far, we have
\begin{equation}\label{dotu}
\|\G a(t)\|_{L^2}+\|\G u(t)\|_{L^2}+\|\G\dot u(t)\|_{L^2}\le C, \quad s\in(0,1/2].
\end{equation}
According to \eqref{a.2}, we have
\begin{equation}\label{001}
\|\nabla^{l+1}f\|_{L^2}\ge C\|\G f\|_{L^2}^{-\frac{1}{l+s}}\|\nabla^lf\|_{L^2}^{1+\frac{1}{l+s}}.
\end{equation}
By this fact and the uniform bound \eqref{uniform-bound} and \eqref{dotu}, we deduce
\begin{equation*}
\begin{aligned}
\|\nabla a\|_{L^2}^2+\|\nabla u\|_{L^2}^2+\|\nabla^2u\|_{L^2}^2+\|\nabla \dot u\|_{L^2}^2
\ge C(\|a\|_{H^1}^2+\|u\|_{H^1}^2+\|\dot u\|_{L^2}^2)^{1+\frac1s}.
\end{aligned}
\end{equation*}
Thus, thanks to Proposition \ref{pro1} and notice that the estimate \eqref{est2} has nothing to do with the initial data, we obtain
\begin{equation*}
\frac{d}{dt}(\|a\|_{H^1}^2+\|u\|_{H^1}^2+\|\dot u\|_{L^2}^2)+C(\|a\|_{H^1}^2+\|u\|_{H^1}^2+\|\dot u\|_{L^2}^2)^{1+\frac1s}\le 0.
\end{equation*}
Solving this inequality directly, and using \eqref{uniform-bound}, we get
\begin{equation*}
\|a(t)\|_{H^1}^2+\|u(t)\|_{H^1}^2+\|\dot u(t)\|_{L^2}^2\le C(1+t)^{-s},
\end{equation*}
which completes the proof.
\end{proof}

Before deriving the propagation of the negative Sobolev norms of the solution for $s\in(1/2,3/2)$, we need to improve the decay estimate of the first order spatial derivative of the solution for $s\in(0,1/2]$.
And hence, we need to introduce the following lemma.
\begin{lemm}\label{diss}
Under the assumptions of Theorem \ref{Thm2}, and $s\in(0,1/2]$, we define
\begin{equation*}
\mathcal{E}^2_1(t)\stackrel{\Delta}{=}\|\nabla u\|^2_{H^1}+P'(1)\|\nabla a\|^2_{H^1}+2\delta_0\int\nabla u\cdot\nabla^2a dx.
\end{equation*}
Then there exists a large time $T_1$, such that
\begin{equation}\label{1111}
\frac{d}{dt}\mathcal{E}^2_1(t)+c_* (\|\nabla^2 u\|_{H^1}^2+\|\nabla^2 a\|_{L^2}^2) \lm 0
\end{equation}
holds on for all $t\ge T_1$.
Here $c_*=\min{\{\mu,\delta_0P'(1)}\}$, and $\delta_0$ is a small constant.
\end{lemm}

\begin{proof}
To obtain \eqref{1111}, we need to introduce the following energy estimate coming from Lemma 2.4 in \cite{gao-wei-yao}
\begin{equation*}
\begin{split}
&\frac{d}{dt}\left\{\frac12\|\nabla u\|_{H^1}^2+\frac{P'(1)}{2}\|\nabla a\|_{H^1}^2
  +\delta_0 \int\nabla u\cdot\nabla^2a dx \right\}
+\frac{3\mu}{4}\|\nabla^2 u\|_{H^1}^2 +\frac{3\delta_0}{4}P'(1)\|\nabla^2 a\|_{L^2}^2\\
&\lm C(\|a\|^{\frac14}_{L^2}+\|u\|^{\frac14}_{L^2}+\|\nabla u\|^{\frac14}_{L^2}+\|a \|_{H^1}+\|u \|_{H^1})
       (\|\nabla^2 u\|^2_{H^1}+\|\nabla^2 a\|^2_{L^2}).
\end{split}
\end{equation*}
According to the decay result \eqref{prop1est} obtained in Proposition \ref{prop1}, one may conclude that
\begin{equation*}
\|a\|^{\frac14}_{L^2}+\|u\|^{\frac14}_{L^2}+\|\nabla u\|^{\frac14}_{L^2}+\|a \|_{H^1}+\|u \|_{H^1}\le C(1+t)^{-\frac s8},
\end{equation*}
and thus, there exists a large time $T_1>0$ such that
\begin{equation*}
\|a\|^{\frac14}_{L^2}+\|u\|^{\frac14}_{L^2}+\|\nabla u\|^{\frac14}_{L^2}+\|a \|_{H^1}+\|u \|_{H^1}\le \frac14\min\{\mu,\delta_0P'(1)\}
\end{equation*}
holds on for all $t\ge T_1$. Therefore, we obtain
\begin{equation*}
\frac{d}{dt}\{\|\nabla u\|_{H^1}^2+P'(1)\|\nabla a\|_{H^1}^2
  +2\delta_0\int\nabla u\cdot\nabla^2a dx \}+\mu\|\nabla^2 u\|_{H^1}^2 +\delta_0P'(1)\|\nabla^2 a\|_{L^2}^2\le 0.
\end{equation*}
Taking $c_*=\min{\{\mu,\delta_0P'(1)}\}$, it holds on
\begin{equation*}
\frac{d}{dt}\mathcal{E}^2_1(t)+c_* (\|\nabla^2 u\|_{H^1}^2+\|\nabla^2 a\|_{L^2}^2) \lm 0.
\end{equation*}
By virtue of the smallness of $\delta_0$, there exist two constants $c_3$ and $C_3$(independent of time), such that
\begin{equation}\label{E1}
c_3(\|\nabla u\|^2_{H^1}+\|\nabla a\|^2_{H^1})\lm\mathcal{E}^2_1(t)\lm C_3(\|\nabla u\|^2_{H^1}+\|\nabla a\|^2_{H^1}).
\end{equation}
Therefore, we complete the proof of this lemma.
\end{proof}

With Lemma \ref{diss} at hand, we can improve the decay estimate of the first derivative of the solution, and also establish the decay estimate of the second derivative of the solution when $s\in(0,1/2]$.

\begin{prop}\label{prop5}
Under the assumptions of Theorem \ref{Thm2}, and $s\in(0,1/2]$, then, it holds on for all $t\ge T_1$
\begin{equation}\label{hidecay}
\|\nabla u(t)\|_{H^1}^2+\|\nabla a(t)\|_{H^1}^2\le C(1+t)^{-(1+s)},
\end{equation}
where $C$ is a constant independent of time, and $T_1$ is given in Lemma \ref{diss}.
\end{prop}
\begin{proof}
According to \eqref{a.2}, it is easy to check that
\begin{equation*}
\|\nabla^{l+1}f\|_{L^2}\ge C\|\G f\|_{L^2}^{-\frac{1}{l+s}}\|\nabla^lf\|_{L^2}^{1+\frac{1}{l+s}},
\end{equation*}
this together with \eqref{uniform-bound} and \eqref{dotu} gives rise to
\begin{equation*}
\|\nabla^2 u\|^2_{H^1}+\|\nabla^2 a\|^2_{L^2}\ge C(\|\nabla u\|^2_{H^1}+\|\nabla a\|^2_{H^1})^{1+\frac{1}{1+s}}.
\end{equation*}
By Lemma \ref{diss} and the equivalent condition \eqref{E1}, we have
\begin{equation*}
\frac{d}{dt}\mathcal{E}^2_1(t)+C(\mathcal{E}^2_1(t))^{1+\frac{1}{1+s}}\le 0 \quad for \quad t\ge T_1.
\end{equation*}
Solving this inequality directly and using \eqref{uniform-bound}, it holds on
\begin{equation}\label{decay2}
\mathcal{E}^2_1(t)\le \left(\mathcal{E}^2_1(T_1)^{-\frac{2}{1+s}}+\frac{C(t-T_1)}{1+s}\right)^{-(1+s)}\le C(1+t)^{-(1+s)}, \quad t\ge T_1.
\end{equation}
Consequently, the combination of the equivalent condition \eqref{E1} and \eqref{decay2}, yields that for $s\in(0,1/2]$,
\begin{equation*}
\|\nabla u(t)\|_{H^1}^2+\|\nabla a(t)\|_{H^1}^2\le C(1+t)^{-(1+s)}, \quad t\ge T_1,
\end{equation*}
which completes the proof.
\end{proof}

Now we are in the position to establish the decay estimate for the solution of \eqref{cns} with initial data $(a_0,u_0)\in \dot H^{-s}(s\in(1/2,3/2))$.

\begin{prop}\label{prop6}
Under the assumptions of Theorem \ref{Thm2}, when $s\in(1/2,3/2)$, then for all $t\ge0$, we have
\begin{equation*}
\|\G a(t)\|_{L^2}^2+ \|\G u(t)\|_{L^2}^2+ \|\G \dot u(t)\|_{L^2}^2\le C,
\end{equation*}
and
  \begin{equation*}
 \|a(t)\|^2_{H^1}+\|u(t)\|^2_{H^1}+\|\dot u(t)\|_{L^2}^2\le C(1+t)^{-s}.
  \end{equation*}
 For $t\ge T_1$, there holds 
  \begin{equation*}
 \|\nabla a(t)\|_{H^1}^2+\|\nabla u(t)\|_{H^1}^2\le C(1+t)^{-(1+s)},
 \end{equation*}
 where $T_1$ is a fixed large time given in Lemma \ref{diss}, and $C$ is a constant independent of time.
 \end{prop}
 \begin{proof}
 Notice that $a_0, u_0\in\dot H^{-s}\cap L^2 \subset\dot H^{-1/2}$ when $s\in(1/2,3/2)$. Then we derive from what we have proved in \eqref{prop1est} and \eqref{hidecay} with $s=1/2$ that the following decay results:
 \begin{equation}\label{decay3}
 \begin{aligned}
 &\|a(t)\|_{L^2}+\|u(t)\|_{L^2}\le C(1+t)^{-\frac14}, \quad t\ge0,\\&
 \|\nabla a(t)\|_{H^1}+\|\nabla u(t)\|_{H^1}\le C(1+t)^{-\frac34}, \quad t\ge T_1.
 \end{aligned}
 \end{equation}
 Substituting \eqref{decay3} into \eqref{est12}, it is easy to derive that for $s\in(1/2,3/2)$,
 \begin{equation*}
 \begin{aligned}
 &\sup_{\tau\in[0,t]}(\|\G a(\tau)\|^2_{L^2}+\|\G(\rho u)(\tau)\|^2_{L^2})+\int_0^t\|\nabla\G u(\tau)\|^2_{L^2}d\tau\\&
 \quad \lesssim \|\G a_0\|_{L^2}^2+\|\G(\rho_0u_0)\|_{L^2}^2\\&
 \qquad +\int_0^{T_1}(1+\tau)^{-\frac12}(\|\G a\|_{L^2}+\|\G(\rho u)\|_{L^2}+\|\nabla \G u\|_{L^2})d\tau+\int_0^{T_1}(1+\tau)^{-1}d\tau\\&
 \qquad +\int_{T_1}^t(1+\tau)^{\frac12s-\frac74}(\|\G a\|_{L^2}+\|\G(\rho u)\|_{L^2})d\tau+\int_{T_1}^t(1+\tau)^{\frac12s-\frac54}\|\nabla\G u\|_{L^2}d\tau\\&
\qquad +\int_{T_1}^t(1+\tau)^{s-3}d\tau.
 \end{aligned}
 \end{equation*}
 Notice that we can fix the value of $T_1$ according to Lemma \ref{diss}, then direct calculation derives that for all $t\ge0$, there holds
 \begin{equation*}
 \|\G a(t)\|_{L^2}^2+\|\G(\rho u)(t)\|_{L^2}^2\le C \quad for \quad s\in(1/2,3/2).
 \end{equation*}
 Then repeating the progress in the proof of Proposition \ref{prop1} for $s\in(1/2,3/2)$, we obtain that for all $t\ge0$
 \begin{equation*}
 \|\G a(t)\|_{L^2}^2+ \|\G u(t)\|_{L^2}^2+ \|\G \dot u(t)\|_{L^2}^2\le C, \quad s\in(1/2,3/2),
 \end{equation*}
 and
 \begin{equation*}
 \|a(t)\|_{H^1}^2+\|u(t)\|_{H^1}^2+\|\dot u(t)\|_{L^2}^2\le C(1+t)^{-s}, \quad s\in(1/2,3/2).
 \end{equation*}
 Similar to Proposition \ref{prop5} for $s\in(1/2,3/2)$, it is easy to derive that for all $t\ge T_1$
 \begin{equation*}
 \|\nabla a(t)\|_{H^1}^2+\|\nabla u(t)\|_{H^1}^2\le C(1+t)^{-(1+s)}, \quad s\in(1/2, 3/2).
 \end{equation*}
The combination of above two estimates completes the proof.
 \end{proof}

\textbf{Proof of Theorem \ref{Thm2}:}
With the uniform bounds and decay estimates stated in Propositions \ref{prop1}, \ref{prop5} and \ref{prop6},
we can complete the the proof of Theorem \ref{Thm2}.

\section{The proof of the main results for FCNS}\label{4}

Denoting $a:=\rho-1$, $\theta:=T-1$, we rewrite \eqref{fullcns0} in the perturbation form as follows
\begin{equation}\label{fullcns}
  	\left\{\begin{aligned}
 &\partial_ta+\dive u =S_1,\\
 &\partial_tu-\mu\Delta u-(\mu+\lambda)\nabla\dive u+\nabla a+\nabla\theta=S_2,\\
 &\partial_t\theta-\Delta\theta+\dive u=S_3,
  	\end{aligned}\right.
  \end{equation}
where the nonlinear terms $S_1$, $S_2$ and $S_3$ are defined by
\begin{equation}\label{S}
  	\left\{\begin{aligned}
 &S_1:=-a\dive u-u\cdot\nabla a,\\
 &S_2:=-u\cdot\nabla u-h(a)[\mu\Delta u+(\mu+\lambda)\nabla\dive u]+h(a)(\nabla a+\nabla \theta)-g(a)\nabla(a\theta),\\
 &S_3:=-\dive(\theta u)+g(a)(2\mu|D u|^2+\lambda(\dive u)^2)-h(a)\Delta\theta,
  	\end{aligned}\right.
  \end{equation}
  where
  \begin{equation*}
  h(a):=\frac{a}{1+a}, \quad g(a):=\frac{1}{1+a}.
  \end{equation*}

  \subsection{The proof of Theorem \ref{Thm3}.}
 \quad
  First, we shall establish the estimate for the first order spatial derivative of the solution as follows.
  \begin{lemm}\label{l4}
  Under the assumptions of Theorem \ref{0them}, the global solution $(a, u, \theta)$ of Cauchy problem \eqref{fullcns} has the estimate
  \begin{equation}\label{l4est}
  \begin{aligned}
  \frac12&\frac{d}{dt}\int(|\nabla a|^2+|\nabla u|^2+|\nabla\theta|^2)dx+\mu\int|\nabla^2u|^2dx+(\mu+\lambda)\int|\nabla\dive u|^2dx+\int|\nabla^2\theta|^2dx\\&
  \le C(\|(a,u,\theta)\|_{H^1}+\|\nabla(a, u)\|_{L^2}^{\frac12})(\|\nabla^2u\|_{L^2}^2+\|\nabla^2a\|_{L^2}^2+\|\nabla^2\theta\|_{L^2}^2).
  \end{aligned}
  \end{equation}
  \end{lemm}
  \begin{proof}
  Applying $\nabla$ operator to $\eqref{fullcns}_1$, multiplying by $\nabla a$ and integrating over $\mathbb R^3$, we have
  \begin{equation*}
  \frac12\frac{d}{dt}\int |\nabla a|^2dx+\int\nabla\dive u\cdot\nabla adx=\int\nabla S_1\cdot\nabla adx.
  \end{equation*}
  Integrating by part and applying H\"{o}lder inequality, we obtain
    \begin{equation}\label{m1}
  \frac12\frac{d}{dt}\int |\nabla a|^2dx-\int\nabla u\cdot\nabla ^2adx\le \|S_1\|_{L^2}\|\nabla^2a\|_{L^2}.
  \end{equation}
   Direct calculation gives
  \begin{equation}\label{n8}
  \|S_1\|_{L^2}\le \|a\|_{L^3}\|\dive u\|_{L^6}+\|u\|_{L^3}\|\nabla a\|_{L^6}\le C(\|a\|_{H^1}+\|u\|_{H^1})(\|\nabla^2u\|_{L^2}+\|\nabla^2a\|_{L^2}),
  \end{equation}
this together with \eqref{m1} yields
\begin{equation}\label{m7}
\frac12\frac{d}{dt}\int |\nabla a|^2dx-\int\nabla u\cdot\nabla ^2adx\le C(\|a\|_{H^1}+\|u\|_{H^1})(\|\nabla^2u\|_{L^2}^2+\|\nabla^2a\|_{L^2}^2).
\end{equation}

Second, applying $\nabla$ operator to $\eqref{fullcns}_2$, and then multiplying by $\nabla u$ and integrating over $\mathbb R^3$, we have
\begin{equation*}
\frac12
\frac{d}{dt}\int|\nabla u|^2dx+\mu\int|\nabla^2u|^2dx
+\!(\mu\!+\lambda)\!\!\int|\nabla\dive u|^2dx
\!+\!\!\int\nabla^2a\cdot\nabla udx\!+\!\!\int\nabla^2\theta\cdot\nabla udx
=\!\int\nabla S_2\cdot\nabla udx,
\end{equation*}
which yields directly
\begin{equation*}
\begin{aligned}
\frac12&\frac{d}{dt}\!\!\int|\nabla u|^2dx+\mu\!\!\int|\nabla^2u|^2dx
\!+(\mu+\lambda)\!\!\int|\nabla\dive u|^2dx+\!\!\int\nabla^2a\cdot\nabla udx
\!+\!\!\int\nabla^2\theta\cdot\nabla udx
\le \|S_2\|_{L^2}\|\nabla^2u\|_{L^2}.
\end{aligned}
\end{equation*}
Using Sobolev inequality, it is easy to check that
\begin{equation}\label{m2}
\|u\cdot\nabla u\|_{L^2}\le\|u\|_{L^3}\|\nabla u\|_{L^6}\le C\|u\|_{H^1}\|\nabla^2u\|_{L^2}.
\end{equation}
Recalling the lower bound of density, and using Sobolev inequality and \eqref{uniform}, we obtain
\begin{equation}
\|\frac{a}{1+a}[\mu\Delta u+(\mu+\lambda)\dive u]\|_{L^2}
\le C\|a\|_{L^\infty}\|\nabla^2u\|_{L^2}
\le C\|\nabla a\|_{L^2}^{\frac12}\|\nabla^2a\|_{L^2}^{\frac12}\|\nabla^2u\|_{L^2}
\le C\|\nabla a\|_{L^2}^{\frac12}\|\nabla^2u\|_{L^2}.
\end{equation}
Similarly, it holds on
\begin{equation}\label{m3}
\|\frac{a}{1+a}(\nabla a+\nabla\theta)\|_{L^2}\le C\|a\|_{L^3}(\|\nabla a\|_{L^6}+\|\nabla\theta\|_{L^6})\le C
\|a\|_{H^1}(\|\nabla^2a\|_{L^2}+\|\nabla^2\theta\|_{L^2}),
\end{equation}
and
\begin{equation}\label{m4}
\|\frac{1}{1+a}\nabla(a\theta)\|_{L^2}
\le C(\|a\|_{L^3}\|\nabla\theta\|_{L^6}+\|\theta\|_{L^3}\|\nabla a\|_{L^6})
\le C(\|a\|_{H^1}+\|\theta\|_{H^1})(\|\nabla^2\theta\|_{L^2}+\|\nabla^2a\|_{L^2}).
\end{equation}
The combination of \eqref{m2}-\eqref{m4} gives
\begin{equation*}
\|S_2\|_{L^2}\|\nabla^2u\|_{L^2}
\le C(\|u\|_{H^1}+\|a\|_{H^1}+\|\theta\|_{H^1}
+\|\nabla a\|_{L^2}^{\frac12})(\|\nabla^2\theta\|_{L^2}^2+\|\nabla^2a\|_{L^2}^2+\|\nabla^2u\|_{L^2}^2),
\end{equation*}
which implies
\begin{equation}\label{m8}
\begin{aligned}
\frac12&\frac{d}{dt}\int|\nabla u|^2dx+\mu\int|\nabla^2u|^2dx+(\mu+\lambda)\int|\nabla\dive u|^2dx+\int\nabla^2a\cdot\nabla udx+\int\nabla^2\theta\cdot\nabla udx\\&
\le C(\|u\|_{H^1}+\|a\|_{H^1}+\|\theta\|_{H^1}+\|\nabla a\|_{L^2}^{\frac12})(\|\nabla^2\theta\|_{L^2}^2+\|\nabla^2a\|_{L^2}^2+\|\nabla^2u\|_{L^2}^2).
\end{aligned}
\end{equation}

Last, applying $\nabla$ operator to $\eqref{fullcns}_3$,
multiplying by $\nabla \theta$ and integrating over $\mathbb R^3$, we have
\begin{equation*}
\frac12\frac{d}{dt}\int|\nabla\theta|^2dx+\int|\nabla^2\theta|^2dx+\int\nabla\dive u\cdot\nabla\theta dx=\int\nabla S_3\cdot\nabla\theta dx.
\end{equation*}
which integrating by part yields directly
\begin{equation}\label{m5}
\frac12\frac{d}{dt}\int|\nabla\theta|^2dx+\int|\nabla^2\theta|^2dx-\int\nabla u\cdot\nabla^2\theta dx\le \|S_3\|_{L^2}\|\nabla^2\theta\|_{L^2}.
\end{equation}
According to \eqref{uniform} and routine calculation, it is easy to check that
\begin{equation}\label{p2}
\begin{aligned}
\|S_3\|_{L^2}&\le\|\theta\|_{L^3}\|\dive u\|_{L^6}+\|u\|_{L^3}\|\nabla\theta\|_{L^6}+\|\nabla u\|_{L^3}\|\nabla u\|_{L^6}+\|a\|_{L^\infty}\|\nabla^2\theta\|_{L^2}\\&
\le C(\|\theta\|_{H^1}+\|u\|_{H^1})(\|\nabla^2u\|_{L^2}+\|\nabla^2\theta\|_{L^2})+C\|\nabla u\|_{L^2}^{\frac12}\|\nabla^2u\|_{L^2}^{\frac12}\|\nabla^2u\|_{L^2}\\&
\quad +C\|\nabla a\|_{L^2}^{\frac12}\|\nabla^2a\|_{L^2}^{\frac12}\|\nabla^2\theta\|_{L^2}\\&
\le C(\|\theta\|_{H^1}+\|u\|_{H^1}+\|\nabla u\|_{L^2}^{\frac12}+\|\nabla a\|_{L^2}^{\frac12})(\|\nabla^2u\|_{L^2}+\|\nabla^2\theta\|_{L^2}),
\end{aligned}
\end{equation}
which implies that
\begin{equation}\label{m6}
\|S_3\|_{L^2}\|\nabla^2\theta\|_{L^2}\le C(\|\theta\|_{H^1}+\|u\|_{H^1}+\|\nabla u\|_{L^2}^{\frac12}+\|\nabla a\|_{L^2}^{\frac12})(\|\nabla^2u\|_{L^2}^2+\|\nabla^2\theta\|^2_{L^2}).
\end{equation}
The combination of \eqref{m5} and \eqref{m6} gives immediately
\begin{equation}\label{m9}
\begin{aligned}
\frac12&\frac{d}{dt}\int|\nabla\theta|^2dx+\int|\nabla^2\theta|^2dx-\int\nabla u\cdot\nabla^2\theta dx\\&
\le C(\|\theta\|_{H^1}+\|u\|_{H^1}+\|\nabla u\|_{L^2}^{\frac12}+\|\nabla a\|_{L^2}^{\frac12})(\|\nabla^2u\|^2_{L^2}+\|\nabla^2\theta\|^2_{L^2}).
\end{aligned}
\end{equation}
Summing \eqref{m7}, \eqref{m8} and \eqref{m9}, we obtain
\begin{equation*}
  \begin{aligned}
  \frac12&\frac{d}{dt}\int(|\nabla a|^2+|\nabla u|^2+|\nabla\theta|^2)dx+\mu\int|\nabla^2u|^2dx+(\mu+\lambda)\int|\nabla\dive u|^2dx+\int|\nabla^2\theta|^2dx\\&
  \le C(\|(a,u,\theta)\|_{H^1}
  +\|\nabla(a, u)\|_{L^2}^{\frac12})(\|\nabla^2u\|_{L^2}^2+\|\nabla^2a\|_{L^2}^2+\|\nabla^2\theta\|_{L^2}^2).
  \end{aligned}
  \end{equation*}
Therefore, we completes the proof.
\end{proof}

Next, we shall establish the energy estimate for the second order spatial derivative of the solution.

\begin{lemm}\label{l5}
  Under the assumptions of Theorem \ref{0them}, the global solution $(a, u, \theta)$ of Cauchy problem  \eqref{fullcns} has the estimate
 \begin{equation}\label{l5est}
\begin{aligned}
\frac12&\frac{d}{dt}\int(|\nabla^2a|^2+|\nabla^2u|^2+|\nabla^2\theta|^2)dx+\mu\int|\nabla^3u|^2dx+(\mu+\lambda)\int|\nabla^2\dive u|^2dx+\int|\nabla^3\theta|^2dx\\&
\le C(\|\nabla u\|_{L^2}^{\frac14}+\|\nabla (u,a)\|_{L^2}^{\frac12}+\|(u,\theta)\|_{H^1}+\|\theta\|_{L^2}^{\frac14})(\|\nabla^2u\|_{H^1}^2+\|\nabla^2\theta\|_{H^1}^2+\|\nabla^2a\|_{L^2}^2).
\end{aligned}
\end{equation}
  \end{lemm}
  \begin{proof}
  First, applying $\nabla^2$ operator to $\eqref{fullcns}_1$, and then multiplying by $\nabla^2a$ and integrating over $\mathbb R^3$, we get
\begin{equation}\label{s1}
\frac12\frac{d}{dt}\int|\nabla^2a|^2dx+\int\nabla^2\dive u\cdot\nabla^2adx\le \|\nabla^2S_1\|_{L^2}\|\nabla^2a\|_{L^2}.
\end{equation}
Recall that $S_1=-a\dive u-u\cdot\nabla a$, direct calculation shows
$$\nabla^2(a\dive u)=a\nabla^2\dive u+2\nabla a\nabla\dive u+\nabla^2a\dive u,$$
it follows that
\begin{equation}\label{s2}
\begin{aligned}
\|&\nabla^2(a\dive u)\|_{L^2}\|\nabla^2a\|_{L^2}\\&
\le (\|a\|_{L^\infty}\|\nabla^2\dive u\|_{L^2}+\|\nabla a\|_{L^3}\|\nabla\dive u\|_{L^6})\|\nabla^2a\|_{L^2}+\|\dive u\|_{L^\infty}\|\nabla^2a\|_{L^2}^2\\&
\le C(\|a\|_{L^\infty}+\|\nabla a\|_{L^3})(\|\nabla^3u\|_{L^2}^2+\|\nabla^2a\|_{L^2}^2)+\|\dive u\|_{L^\infty}\|\nabla^2a\|_{L^2}^2.
\end{aligned}
\end{equation}
By routine checking, it holds on
$$\nabla^2(u\cdot\nabla a)=u\cdot\nabla(\nabla^2a)+2\nabla u\cdot\nabla^2a+\nabla^2u\cdot\nabla a.$$
Integrating by part, we get
$$\int u\cdot\nabla(\nabla^2a)\cdot\nabla^2adx=\int u\cdot\nabla(\frac12|\nabla^2a|^2)dx=-\frac12\int(\dive u)|\nabla^2a|^2dx,$$
and hence, we obtain
\begin{equation}\label{s3}
\begin{aligned}
|\int\nabla^2(u\cdot\nabla a)\cdot\nabla^2adx|&\le C(\|\nabla u\|_{L^\infty}\|\nabla^2a\|_{L^2}+\|\nabla^2u\|_{L^6}\|\nabla a\|_{L^3})\|\nabla^2a\|_{L^2}\\&\
\quad +C\|\dive u\|_{L^\infty}\|\nabla^2a\|_{L^2}^2\\&
\le C\|\nabla u\|_{L^\infty}\|\nabla^2a\|_{L^2}^2+C\|\nabla a\|_{L^3}\|\nabla^3u\|_{L^2}\|\nabla^2a\|_{L^2}.
\end{aligned}
\end{equation}
The combination of \eqref{s1}, \eqref{s2} and \eqref{s3} yields
\begin{equation*}
|\int\nabla^2S_1\cdot\nabla^2adx|\le C(\|\nabla a\|_{L^3}+\|a\|_{L^\infty}+\|\nabla u\|_{L^\infty})\|\nabla^2a\|_{L^2}^2+C(\|\nabla a\|_{L^3}+\|a\|_{L^\infty})\|\nabla^3u\|_{L^2}^2.
\end{equation*}
Thanks to the Sobolev inequality and the uniform estimate \eqref{uniform}, it follows that
$$\|a\|_{L^\infty}+\|\nabla a\|_{L^3}\le \|\nabla a\|_{L^2}^{\frac12}\|\nabla^2a\|_{L^2}^{\frac12}\le C\|\nabla a\|_{L^2}^{\frac12},$$
and
\begin{equation}\label{df}
\|\nabla u\|_{L^\infty}\|\nabla^2a\|_{L^2}^2\le C\|\nabla u\|_{L^2}^{\frac14}\|\nabla^3u\|_{L^2}^{\frac34}\|\nabla^2a\|_{L^2}^{\frac54}\|\nabla^2a\|_{L^2}^{\frac34}\le C\|\nabla u\|_{L^2}^{\frac14}(\|\nabla^3u\|_{L^2}^2+\|\nabla^2a\|_{L^2}^2).
\end{equation}
The combination of above estimates yields
\begin{equation}\label{n5}
\frac12\frac{d}{dt}\int |\nabla^2a|^2dx-\int\nabla^2u\cdot\nabla^3adx\le (\|\nabla a\|_{L^2}^{\frac12}+\|\nabla u\|_{L^2}^{\frac14})(\|\nabla^2a\|_{L^2}^2+\|\nabla^3u\|_{L^2}^2).
\end{equation}

Second, applying $\nabla^2$ operator to $\eqref{fullcns}_2$, multiplying by $\nabla^2u$
and integrating over $\mathbb R^3$, we get
\begin{equation}\label{n1}
\begin{aligned}
\frac12&\frac{d}{dt}\int|\nabla^2u|^2dx+\mu\int|\nabla^3u|^2dx+(\mu+\lambda)\int|\nabla^2\dive u|^2dx+\int\nabla^3a\cdot\nabla^2udx+\int\nabla^3\theta\cdot\nabla^2udx\\&
=\int\nabla^2S_2\cdot\nabla^2udx=-\int\nabla S_2\cdot\nabla^3udx.
\end{aligned}
\end{equation}
By routine checking, it is easy to see
\begin{equation*}
\begin{aligned}
\nabla S_2&=-\nabla u\cdot\nabla u-u\cdot\nabla^2 u-\frac{a}{1+a}[\mu\nabla\Delta u+(\mu+\lambda)\nabla^2\dive u]+\frac{\nabla a}{(1+a)^2}[\mu\Delta u+(\mu+\lambda)\nabla\dive u]\\&
\quad -\frac{a}{1+a}(\nabla^2a+\nabla^2\theta)+\frac{\nabla a}{(1+a)^2}(\nabla a+\nabla\theta)-\frac{1}{1+a}\nabla^2(a\theta)+\frac{\nabla a}{(1+a)^2}\nabla(a\theta).
\end{aligned}
\end{equation*}
By using Sobolev inequality, it holds on
\begin{equation*}
\begin{aligned}
\|\nabla S_2\|_{L^2}&\le\|\nabla u\|_{L^3}\|\nabla u\|_{L^6}+\|u\|_{L^3}\|\nabla^2u\|_{L^6}+\|a\|_{L^\infty}\|\nabla^3u\|_{L^2}+\|\nabla a\|_{L^3}\|\nabla^2u\|_{L^6}\\&
\quad +\|a\|_{L^\infty}(\|\nabla^2a\|_{L^2}+\|\nabla^2\theta\|_{L^2})+\|\nabla a\|_{L^3}(\|\nabla a\|_{L^6}+\|\nabla\theta\|_{L^6})+\|\theta\|_{L^\infty}\|\nabla^2a\|_{L^2}\\&
\quad +\|a\|_{L^\infty}\|\nabla^2\theta\|_{L^2}+\|\theta\|_{L^6}\|(\nabla a)^2\|_{L^3}+\|a\|_{L^\infty}\|\nabla a\|_{L^3}\|\nabla\theta\|_{L^6}\\&
\le C(\|\nabla (u,a)\|_{L^3}+\|u\|_{L^3}+\|a\|_{L^\infty}+\|a\|_{L^\infty}\|\nabla a\|_{L^3})(\|\nabla ^2u\|_{H^1}+\|\nabla^2a\|_{L^2}+\|\nabla^2\theta\|_{L^2})\\&
\quad +\|\theta\|_{L^\infty}\|\nabla^2a\|_{L^2}+C\|\nabla\theta\|_{L^2}\|\nabla^2a\|_{L^2}^2.
\end{aligned}
\end{equation*}
By virtue of Sobolev inequality and the uniform bound \eqref{uniform}, we see
\begin{equation*}
\begin{aligned}
\|&\nabla (u,a)\|_{L^3}+\|u\|_{L^3}+\|a\|_{L^\infty}+\|a\|_{L^\infty}\|\nabla a\|_{L^3}\\&
\le C(\|\nabla u\|_{L^2}^{\frac12}\|\nabla^2u\|_{L^2}^{\frac12}+\|u\|_{H^1}+\|\nabla a\|_{L^2}^{\frac12}\|\nabla^2a\|_{L^2}^{\frac12}+\|\nabla a\|_{L^2}^{\frac12}\|\nabla^2a\|_{L^2}^{\frac12}\|\nabla a\|_{H^1})\\&
\le C(\|\nabla u\|_{L^2}^{\frac12}+\|u\|_{H^1}+\|\nabla a\|_{L^2}^{\frac12}),
\end{aligned}
\end{equation*}
and
$$\|\nabla\theta\|_{L^2}\|\nabla^2a\|_{L^2}^2\le C\|\nabla\theta\|_{L^2}\|\nabla^2a\|_{L^2}.$$
Hence, we have
\begin{equation}\label{n2}
\|\nabla S_2\|_{L^2}
\le C(\|\nabla u\|_{L^2}^{\frac12}\!+\!\|u\|_{H^1}
\!+\!\|\nabla a\|_{L^2}^{\frac12}\!+\!\|\nabla\theta\|_{L^2})
(\|\nabla ^2u\|_{H^1}
\!+\!\|\nabla^2a\|_{L^2}+\|\nabla^2\theta\|_{L^2})
\!+\!\|\theta\|_{L^\infty}\|\nabla^2a\|_{L^2}.
\end{equation}
The combination of \eqref{n1} and \eqref{n2} gives rise to
\begin{equation}\label{n3}
\begin{aligned}
\frac12&\frac{d}{dt}\int|\nabla^2u|^2dx+\mu\int|\nabla^3u|^2dx+(\mu+\lambda)\int|\nabla^2\dive u|^2dx+\int\nabla^3a\cdot\nabla^2udx+\int\nabla^3\theta\cdot\nabla^2udx\\&
\le C(\|\nabla u\|_{L^2}^{\frac12}\!+\!\|u\|_{H^1}+\|\nabla a\|_{L^2}^{\frac12}\!+\!\|\nabla\theta\|_{L^2})
(\|\nabla ^2u\|_{H^1}^2\!+\!\|\nabla^2a\|_{L^2}^2\!+\!\|\nabla^2\theta\|_{L^2}^2)
 \!+\!\|\theta\|_{L^\infty}\|\nabla^2a\|_{L^2}\|\nabla^3u\|_{L^2}.
\end{aligned}
\end{equation}
For the last term on the right-hand side of above inequality, using Cauchy inequality,
Sobolev inequality, and the uniform bound \eqref{uniform}, we have
\begin{equation}\label{df2}
\begin{aligned}
\|\theta\|_{L^\infty}\|\nabla^2a\|_{L^2}\|\nabla^3u\|_{L^2}&\le\|\theta\|_{L^2}^{\frac14}\|\nabla^2\theta\|_{L^2}^{\frac34}\|\nabla^2a\|_{L^2}^{\frac14}\|\nabla^2a\|_{L^2}^{\frac34}\|\nabla^3u\|_{L^2}\\&
\le \|\theta\|_{L^2}^{\frac14}(\|\nabla^2\theta\|_{L^2}+\|\nabla^2a\|_{L^2})\|\nabla^3u\|_{L^2}\\&
\le \|\theta\|_{L^2}^{\frac14}(\|\nabla^2\theta\|_{L^2}^2+\|\nabla^2a\|_{L^2}^2+\|\nabla^3u\|_{L^2}^2),
\end{aligned}
\end{equation}
this together with \eqref{n3} gives
\begin{equation}\label{n6}
\begin{aligned}
\frac12&\frac{d}{dt}\int|\nabla^2u|^2dx+\mu\int|\nabla^3u|^2dx+(\mu+\lambda)\int|\nabla^2\dive u|^2dx+\int\nabla^3a\cdot\nabla^2udx+\int\nabla^3\theta\cdot\nabla^2udx\\&
\le C(\|\nabla(u,a)\|_{L^2}^{\frac12}+\|u\|_{H^1}+\|\nabla\theta\|_{L^2}+\|\theta\|_{L^2}^{\frac14})(\|\nabla ^2u\|_{H^1}^2+\|\nabla^2a\|_{L^2}^2+\|\nabla^2\theta\|_{L^2}^2).
\end{aligned}
\end{equation}

Last, applying $\nabla^2$ operator to $\eqref{fullcns}_3$, and then multiplying by $\nabla^2\theta$ and integrating over $\mathbb R^3$, we get
\begin{equation*}
\frac12\frac{d}{dt}\int|\nabla^2\theta|^2dx+\int|\nabla^3\theta|^2dx+\int\nabla^2\dive u\cdot\nabla^2\theta dx=\int\nabla^2S_3\cdot\nabla^2\theta dx.
\end{equation*}
Integrating by part, we get
\begin{equation}\label{n4}
\frac12\frac{d}{dt}\int|\nabla^2\theta|^2dx+\int|\nabla^3\theta|^2dx-\int\nabla^2 u\cdot\nabla^3\theta dx\le \|\nabla S_3\|_{L^2}\|\nabla^3\theta\|_{L^2}.
\end{equation}
It is easy to compute
\begin{equation*}
\begin{aligned}
\|\nabla S_3\|_{L^2}&\le\|\nabla u\|_{L^3}\|\nabla\theta\|_{L^6}+\|\theta\|_{L^3}\|\nabla^2u\|_{L^6}+\|u\|_{L^3}\|\nabla^2\theta\|_{L^6}+\|\nabla u\|_{L^3}\|\nabla^2u\|_{L^6}\\&
\quad +\|\nabla a\|_{L^6}\|(\nabla u)^2\|_{L^3}+\|a\|_{L^\infty}\|\nabla^3\theta\|_{L^2}+\|\nabla a\|_{L^3}\|\nabla^2\theta\|_{L^6}\\&
\le C(\|\nabla u\|_{L^3}+\|\theta\|_{L^3}+\|u\|_{L^3}+\|a\|_{L^\infty}+\|\nabla a\|_{L^3})(\|\nabla^2\theta\|_{H^1}+\|\nabla^3 u\|_{L^2})\\&
\quad +\|\nabla u\|_{L^\infty}\|\nabla u\|_{L^3}\|\nabla a\|_{L^6}.
\end{aligned}
\end{equation*}
Using Sobolev inequality and the uniform bound \eqref{uniform}, we have
\begin{equation*}
\begin{aligned}
\|\nabla& u\|_{L^3}+\|\theta\|_{L^3}+\|u\|_{L^3}+\|a\|_{L^\infty}+\|\nabla a\|_{L^3}\\&
\le C(\|\nabla u\|_{L^2}^{\frac12}\|\nabla^2u\|_{L^2}^{\frac12}+\|\theta\|_{H^1}+\|u\|_{H^1}+\|\nabla a\|_{L^2}^{\frac12}\|\nabla^2a\|_{L^2}^{\frac12})\\&
\le C(\|\nabla u\|_{L^2}^{\frac12}+\|\nabla a\|_{L^2}^{\frac12}+\|\theta\|_{H^1}+\|u\|_{H^1}),
\end{aligned}
\end{equation*}
and
\begin{equation*}
\begin{aligned}
\|\nabla u\|_{L^\infty}\|\nabla u\|_{L^3}\|\nabla a\|_{L^6}&\le \|\nabla^2u\|_{L^2}^{\frac12}\|\nabla^3u\|_{L^2}^{\frac12}\|\nabla u\|_{L^2}^{\frac12}|\nabla^2 u\|_{L^2}^{\frac12}\|\nabla^2 a\|_{L^2}\\&
\le C\|\nabla u\|_{L^2}^{\frac12}(\|\nabla^2u\|_{L^2}+\|\nabla^3u\|_{L^2}).
\end{aligned}
\end{equation*}
Thus, it follows that
\begin{equation*}
\begin{aligned}
\|\nabla S_3\|_{L^2}\|\nabla^3\theta\|_{L^2}&\le C(\|\nabla u\|_{L^2}^{\frac12}+\|\nabla a\|_{L^2}^{\frac12}+\|\theta\|_{H^1}+\|u\|_{H^1})(\|\nabla^2\theta\|^2_{H^1}+\|\nabla^2u\|_{H^1}^2).
\end{aligned}
\end{equation*}
Substituting above two estimates into \eqref{n4}, we obtain
\begin{equation}\label{n7}
\begin{aligned}
\frac12&\frac{d}{dt}\int|\nabla^2\theta|^2dx+\int|\nabla^3\theta|^2dx-\int\nabla^2 u\cdot\nabla^3\theta dx\\&
\le C(\|\nabla u\|_{L^2}^{\frac12}+\|\nabla a\|_{L^2}^{\frac12}+\|\theta\|_{H^1}+\|u\|_{H^1})(\|\nabla^2\theta\|^2_{H^1}+\|\nabla^2u\|_{H^1}^2).
\end{aligned}
\end{equation}
Finally, the combination of \eqref{n5}, \eqref{n6} and \eqref{n7} gives rise to
\begin{equation*}
\begin{aligned}
\frac12&\frac{d}{dt}\int(|\nabla^2a|^2+|\nabla^2u|^2+|\nabla^2\theta|^2)dx+\mu\int|\nabla^3u|^2dx+(\mu+\lambda)\int|\nabla^2\dive u|^2dx+\int|\nabla^3\theta|^2dx\\&
\le C(\|\nabla u\|_{L^2}^{\frac14}+\|\nabla (u,a)\|_{L^2}^{\frac12}+\|(u,\theta)\|_{H^1}+\|\theta\|_{L^2}^{\frac14})(\|\nabla^2u\|_{H^1}^2+\|\nabla^2\theta\|_{H^1}^2+\|\nabla^2a\|_{L^2}^2),
\end{aligned}
\end{equation*}
which completes the proof of this lemma.
  \end{proof}

  In order to close the estimate, we need to establish the dissipation estimate for $\nabla^2a$.

  \begin{lemm}\label{l6}
  Under the assumptions of Theorem \ref{0them}, the global solution $(a, u, \theta)$ of Cauchy problem \eqref{fullcns} has the estimate
  \begin{equation}\label{l6est}
\begin{aligned}
\frac{d}{dt}&\int\nabla u\cdot\nabla^2adx+\frac34\int|\nabla^2a|^2dx\\&
\le C(\|\nabla^2u\|_{H^1}^2+\|\nabla^2\theta\|_{L^2}^2)+C(\|(a,u,\theta)\|_{H^1}+\|\theta\|_{L^2}^{\frac12})(\|\nabla^2a\|_{L^2}^2+\|\nabla^2\theta\|_{L^2}^2).
\end{aligned}
\end{equation}
\end{lemm}

\begin{proof}
Applying $\nabla$ operator to $\eqref{fullcns}_2$,
multiplying by $\nabla^2a$ and integrating over $\mathbb R^3$, we get
\begin{equation*}
\int\partial_t\nabla u\cdot\nabla^2adx-\int[\mu\nabla\Delta u+(\mu+\lambda)\nabla^2\dive u]\cdot\nabla^2adx
+\int|\nabla^2a|^2dx+\int\nabla^2\theta\cdot\nabla^2adx=\int \nabla S_2\cdot\nabla^2adx.
\end{equation*}
According to $\eqref{fullcns}_1$, it holds on
\begin{equation*}
\begin{aligned}
\int\partial_t\nabla u\cdot\nabla^2adx&=\int\partial_t(\nabla u\cdot\nabla^2a)dx-\int\nabla u\cdot\nabla^2a_tdx\\&
=\frac{d}{dt}\int\nabla u\cdot\nabla^2adx+\int\nabla\dive u\cdot\partial_t(\nabla a)dx\\&
=\frac{d}{dt}\int\nabla u\cdot\nabla^2adx+\int\nabla\dive u\cdot(\nabla S_1-\nabla\dive u)dx.
\end{aligned}
\end{equation*}
Then using Cauchy and H\"{o}lder inequalities, we obtain
\begin{equation*}
\begin{aligned}
\frac{d}{dt}&\int\nabla u\cdot\nabla^2adx+\int|\nabla^2a|^2dx\\&
=\int[\mu\nabla\Delta u+(\mu+\lambda)\nabla^2\dive u]\cdot\nabla^2adx+\int \nabla S_2\cdot\nabla^2adx\\&
\quad -\int\nabla\dive u\cdot(\nabla S_1-\nabla\dive u)dx-\int\nabla^2\theta\cdot\nabla^2adx\\&
\le \frac14\|\nabla^2a\|_{L^2}^2+C(\|\nabla^2u\|_{H^1}^2+\|\nabla^2\theta\|_{L^2}^2)+C(\|\nabla S_2\|_{L^2}^2+\|S_1\|_{L^2}^2).
\end{aligned}
\end{equation*}
According to \eqref{n8}, \eqref{n2} and the uniform bound \eqref{uniform}, it is easy to get
\begin{equation}\label{s5}
\begin{aligned}
\frac{d}{dt}&\int\nabla u\cdot\nabla^2adx+\frac34\int|\nabla^2a|^2dx\\&
\le C(\|\nabla^2u\|_{H^1}^2+\|\nabla^2\theta\|_{L^2}^2)+C(\|a\|_{H^1}+\|u\|_{H^1}+\|\theta\|_{H^1})\|\nabla^2a\|_{L^2}^2+\|\theta\|_{L^\infty}^2\|\nabla^2a\|_{L^2}^2.
\end{aligned}
\end{equation}
Using the Sobolev interpolation and the uniform bound \eqref{uniform} , we have
\begin{equation}\label{s4}
\|\theta\|_{L^\infty}^2\|\nabla^2a\|_{L^2}^2\le \|\theta\|_{L^2}^{\frac12}\|\nabla^2\theta\|_{L^2}^{\frac32}\|\nabla^2a\|_{L^2}^{\frac12}\|\nabla^2a\|_{L^2}^{\frac32}
\le \|\theta\|_{L^2}^{\frac12}(\|\nabla^2\theta\|_{L^2}^2+\|\nabla^2a\|_{L^2}^2).
\end{equation}
Substituting \eqref{s4} into \eqref{s5} derives
\begin{equation*}
\begin{aligned}
\frac{d}{dt}&\int\nabla u\cdot\nabla^2adx+\frac34\int|\nabla^2a|^2dx\\&
\le C(\|\nabla^2u\|_{H^1}^2+\|\nabla^2\theta\|_{L^2}^2)+C(\|a\|_{H^1}+\|u\|_{H^1}+\|\theta\|_{H^1}+\|\theta\|_{L^2}^{\frac12})(\|\nabla^2a\|_{L^2}^2+\|\nabla^2\theta\|_{L^2}^2).
\end{aligned}
\end{equation*}
Thus we complete the proof.
\end{proof}

Combining all the estimates obtained in Lemmas \ref{l4}-\ref{l6}, we derive the following energy estimate.
\begin{lemm}\label{l5}
Under the assumptions of Theorem \ref{0them},  we define
$$\mathcal E_2^2(t)\overset{\Delta}{=}\|\nabla a\|_{H^1}^2+\|\nabla u\|_{H^1}^2+\|\nabla \theta\|_{H^1}^2+\delta\int\nabla u\cdot\nabla^2adx.$$
Then there exists a large time $T_2$, such that
\begin{equation}\label{b1}
\frac{d}{dt}\mathcal E_2^2(t)+c_0\left(\|\nabla^2u\|_{H^1}^2+\|\nabla^2\theta\|_{H^1}^2+\|\nabla^2a\|_{L^2}^2\right)\le 0, \quad t\ge T_2.
\end{equation}
Here $c_0$ is a positive constant, and $\delta$ is a small constant.
\end{lemm}
\begin{proof}
Adding \eqref{l4est} with \eqref{l5est}, we have
\begin{equation}\label{p1}
\begin{aligned}
&\frac12\frac{d}{dt}(\|\nabla a\|_{H^1}^2+\|\nabla u\|_{H^1}^2+\|\nabla \theta\|_{H^1}^2)+\mu\|\nabla^2u\|_{H^1}^2+\|\nabla^2\theta\|_{H^1}^2\\&
\quad \le C(\|\nabla u\|_{L^2}^{\frac14}+\|\nabla(u,a)\|_{L^2}^{\frac12}+\|(a,u,\theta)\|_{H^1}+\|\theta\|_{L^2}^{\frac14})\left(\|\nabla^2u\|_{H^1}^2+\|\nabla^2a\|_{L^2}^2+\|\nabla^2\theta\|_{H^1}^2\right).
\end{aligned}
\end{equation}
Multiplying $\delta$ to \eqref{l6est} and adding with \eqref{p1}, we can choose $\delta$ being small enough to get
\begin{equation}\label{pp11}
\begin{aligned}
\frac12&\frac{d}{dt}\left(\|\nabla a\|_{H^1}^2+\|\nabla u\|_{H^1}^2+\|\nabla \theta\|_{H^1}^2+2\delta\int\nabla u\cdot\nabla^2adx\right)
+\frac{3}{4}(\mu\|\nabla^2u\|_{H^1}^2+\|\nabla^2\theta\|_{H^1}^2)+\frac{3}{4}\delta\|\nabla^2a\|_{L^2}^2\\&
\le C\left(\|\nabla u\|_{L^2}^{\frac14}+\|\nabla(u,a)\|_{L^2}^{\frac12}+\|(a,u,\theta)\|_{H^1}+\|\theta\|_{L^2}^{\frac14}+\|\theta\|_{L^2}^{\frac12}\right)\left(\|\nabla^2u\|_{H^1}^2+\|\nabla^2a\|_{L^2}^2+\|\nabla^2\theta\|_{H^1}^2\right).
\end{aligned}
\end{equation}
Thanks to the decay rate \eqref{decay22} obtained in Theorem \ref{0them}, we obtain
\begin{equation*}
\|\nabla u\|_{L^2}^{\frac14}+\|\nabla(u,a)\|_{L^2}^{\frac12}+\|(a,u,\theta)\|_{H^1}+\|\theta\|_{L^2}^{\frac14}+\|\theta\|_{L^2}^{\frac12}\le C(1+t)^{-\frac{3}{16}}.
\end{equation*}
Thus, there exists a large time $T_2>0$, such that
\begin{equation}\label{cdelta}
C\left(\|\nabla u\|_{L^2}^{\frac14}+\|\nabla(u,a)\|_{L^2}^{\frac12}+\|(a,u,\theta)\|_{H^1}+\|\theta\|_{L^2}^{\frac14}+\|\theta\|_{L^2}^{\frac12}\right)\le \frac14\min\{\mu,1,\delta\}
\end{equation}
holds on for all $t\ge T_2$. Therefore we obtain the following estimate
\begin{equation*}
\frac{d}{dt}\left(\|\nabla a\|_{H^1}^2+\|\nabla u\|_{H^1}^2
+\|\nabla \theta\|_{H^1}^2+2\delta\int\nabla u\cdot\nabla^2adx\right)
+\mu\|\nabla^2u\|_{H^1}^2+\|\nabla^2\theta\|_{H^1}^2+\delta\|\nabla^2a\|_{L^2}^2\le 0
\end{equation*}
Taking $c_0=\min\{\mu,1,\delta\}$, then we get
\begin{equation*}
\frac{d}{dt} \left(\|\nabla a\|_{H^1}^2+\|\nabla u\|_{H^1}^2
+\|\nabla \theta\|_{H^1}^2+\delta\int\nabla u\cdot\nabla^2adx\right)
+c_0\left(\|\nabla^2u\|_{H^1}^2+\|\nabla^2\theta\|_{H^1}^2+\|\nabla^2a\|_{L^2}^2\right)\le 0.
\end{equation*}
By virtue of the smallness of $\delta$, there are two constants $c_4$ and $C_4$ (independent of time) such that
\begin{equation}\label{b4}
c_4(\|\nabla u\|_{H^1}^2+\|\nabla a\|_{H^1}^2+\|\nabla\theta\|_{H^1}^2)\le \mathcal E_2^2(t)\le C_4(\|\nabla u\|_{H^1}^2+\|\nabla a\|_{H^1}^2+\|\nabla\theta\|_{H^1}^2).
\end{equation}
Thus, we complete the proof of this lemma.
\end{proof}

Next, let us establish the following decay estimate, which will give the proof for the Theorem \ref{Thm3}.
\begin{lemm}\label{lemma4}
Under the assumptions of Theorem \ref{0them}, there exists a large time $T_3$, such that
\begin{equation}\label{41}
\|\nabla a\|_{H^1}^2+\|\nabla u\|_{H^1}^2+\|\nabla\theta\|_{H^1}^2+\|\partial_ta\|_{L^2}^2+\|\partial_tu\|_{L^2}^2+\|\partial_t\theta\|_{L^2}^2\le C(1+t)^{-\frac52}
\end{equation}
holds on all $t\ge T_3$, here $C$ is a constant independent of time.
\end{lemm}
\begin{proof}
In order to obtain the time decay rate \eqref{41}, our method
here is to use the Fourier splitting method (by Schonbek\cite{schonbek1}),
which has been applied to obtain decay rate for the incompressible Navier-Stokes equations
in higher order derivative norm (cf.\cite{{schonbek2},{Schonbek-Wiegner}}).
The difficulty, arising from the compressible Navier-Stokes equations, is the appearance of density that obeys the
transport equation rather than diffusive one. To get rid of
this difficulty, our idea is to rewrite the inequality \eqref{b1} as follows
\begin{equation}\label{b2}
\frac{d}{dt}\mathcal E_2^2(t)+\frac{c_0}{2}\left(\|\nabla^2u\|_{H^1}^2+\|\nabla^2\theta\|_{H^1}^2+\|\nabla^2a\|_{L^2}^2\right)+\frac{c_0}{2}\|\nabla^2a\|_{L^2}^2\le 0, \quad t\ge T_2.
\end{equation}
Define $S_0:=\{\xi\in\mathbb R^3||\xi|\le(\frac{R}{1+t})^{\frac12}\}$, then we split the phase space $\mathbb R^3$ into two time-dependent regions, $R$ is a constant defined below. It is easy to see
\begin{equation*}
\begin{aligned}
\int_{\mathbb R^3}|\nabla^3u|^2dx&\ge\int_{\mathbb R^3/S_0}|\xi|^6|\hat u|^2d\xi
\ge\frac{R}{1+t}\int_{\mathbb R^3}|\xi|^4|\hat u|^2d\xi\\&
\ge \frac{R}{1+t}\int_{\mathbb R^3/S_0}|\xi|^4|\hat u|^2d\xi-\frac{R}{1+t}\int_{S_0}|\xi|^4|\hat u|^2d\xi\\&
\ge \frac{R}{1+t}\int_{\mathbb R^3/S_0}|\xi|^4|\hat u|^2d\xi-\frac{R^2}{(1+t)^2}\int_{\mathbb R^3/S_0}|\xi|^2|\hat u|^2d\xi,
\end{aligned}
\end{equation*}
which means
$$\|\nabla^3u\|_{L^2}^2\ge\frac{R}{1+t}\|\nabla^2u\|_{L^2}^2-\frac{R^2}{(1+t)^2}\|\nabla u\|_{L^2}^2.$$
Hence, in the same way, we obtain
\begin{equation}\label{b3}
\begin{aligned}
&\|\nabla^2u\|_{H^1}^2+\|\nabla^2\theta\|_{H^1}^2+\|\nabla^2a\|_{L^2}^2\\
&\ge \frac{R}{1+t}(\|\nabla u\|_{H^1}^2+\|\nabla\theta\|_{H^1}^2+\|\nabla a\|_{L^2}^2)-\frac{R^2}{(1+t)^2}(\|u\|_{H^1}^2+\|\theta\|_{H^1}^2+\|a\|_{L^2}^2).
\end{aligned}
\end{equation}
Substituting \eqref{b3} into \eqref{b2}, we have
\begin{equation*}
\frac{d}{dt}\mathcal E_2^2(t)+\frac{c_0R}{2(1+t)}(\|\nabla u\|_{H^1}^2+\|\nabla\theta\|_{H^1}^2+\|\nabla a\|_{L^2}^2)+\frac{c_0}{2}\|\nabla^2a\|_{L^2}^2
\le \frac{c_0R^2}{2(1+t)^2}(\|u\|_{H^1}^2+\|\theta\|_{H^1}^2+\|a\|_{L^2}^2).
\end{equation*}
The term $\|\nabla^2a\|_{L^2}^2$ on the left-hand side of above inequality plays a role of damping term, hence, it holds on
\begin{equation*}
\frac{d}{dt}\mathcal E_2^2(t)+\frac{c_0R}{2(1+t)}(\|\nabla u\|_{H^1}^2+\|\nabla\theta\|_{H^1}^2+\|\nabla a\|_{H^1}^2)
\le \frac{c_0R^2}{2(1+t)^2}(\|u\|_{H^1}^2+\|\theta\|_{H^1}^2+\|a\|_{L^2}^2),
\end{equation*}
for all $t\ge T_3:=\max\{T_2, R-1\}$.
By the equivalent relation \eqref{b4} and the decay result \eqref{decay2} in Theorem \ref{0them}, we have
\begin{equation*}
\frac{d}{dt}\mathcal E_2^2(t)+\frac{c_0R}{2C_2(1+t)}\mathcal E_2^2(t)\le \frac{c_0R^2}{2}C(1+t)^{-\frac32-2}.
\end{equation*}
Choosing $R=\frac{6C_4}{c_0}$ and multiplying the resulting inequality by $(1+t)^3$, it follows that
\begin{equation*}
\frac{d}{dt}\left[(1+t)^3\mathcal E_2^2(t)\right]\le C(1+t)^{-\frac12}.
\end{equation*}
Integrating the above inequality over $[T_3,t]$, we obtain
\begin{equation*}
\mathcal E_2^2(t)\le (1+t)^{-3}\left[(1+T_3)^3\mathcal E_2^2(T_3)+C(1+t)^{\frac12}-C(1+T_3)^{\frac12}\right],
\end{equation*}
which, together with the uniform bound \eqref{uniform}
and estimate \eqref{thetabound} in Remark \ref{nabla2theta}, it holds on
$$\mathcal E_2^2(t)\le C(1+t)^{-\frac52}.$$
Using the equivalent relation \eqref{b4} again, we have for all $t\ge T_3:=\max\{T_2,\frac{6C_4}{c_0}-1\}$
\begin{equation}\label{p22}
\|\nabla a\|_{H^1}^2+\|\nabla u\|_{H^1}^2+\|\nabla \theta\|_{H^1}^2\le C(1+t)^{-\frac52}.
\end{equation}

Last, according to the equation \eqref{fullcns} and the decay results obtained before, we get
\begin{equation}\label{p3}
\begin{aligned}
\|\partial_ta\|_{L^2}&\le\|\dive u\|_{L^2}+\|S_1\|_{L^2}
\le \|\nabla u\|_{L^2}+C(\|a\|_{H^1}+\|u\|_{H^1})(\|\nabla^2u\|_{L^2}+\|\nabla^2a\|_{L^2})
\le C(1+t)^{-\frac54}.
\end{aligned}
\end{equation}
By \eqref{m2}-\eqref{m4}, we get
\begin{equation}\label{p4}
\begin{aligned}
\|\partial_tu\|_{L^2}&\le \mu\|\Delta u\|_{L^2}+(\mu+\lambda)\|\nabla\dive u\|_{L^2}+\|\nabla a\|_{L^2}+\|\nabla \theta\|_{L^2}+\|S_2\|_{L^2}\\&
\le C(\|\nabla^2u\|_{L^2}+\|\nabla a\|_{L^2}+\|\nabla\theta\|_{L^2})\\&
\quad +C(\|u\|_{H^1}+\|a\|_{L^2}^{\frac14}+\|a\|_{H^1}+\|\theta\|_{H^1})(\|\nabla^2u\|_{L^2}+\|\nabla^2a\|_{L^2}+\|\nabla^2\theta\|_{L^2})\\&
\le C(1+t)^{-\frac54}.
\end{aligned}
\end{equation}
According to \eqref{p2}, we get
\begin{equation}\label{p5}
\begin{aligned}
\|\partial_t\theta\|_{L^2}&\le \|\Delta\theta\|_{L^2}+\|\dive u\|_{L^2}+\|S_3\|_{L^2}\\&
\le \|\nabla^2\theta\|_{L^2}+\|\nabla u\|_{L^2}+C(\|\theta\|_{H^1}+\|u\|_{H^1}+\|\nabla u\|_{L^2}^{\frac12}+\|a\|_{L^2}^{\frac14})(\|\nabla^2u\|_{L^2}+\|\nabla^2\theta\|_{L^2})\\&
\le C(1+t)^{-\frac54}.
\end{aligned}
\end{equation}
The combination of \eqref{p22}-\eqref{p5} completes the proof of this lemma.
\end{proof}

\subsection{The proof of Theorem \ref{1theo}.}

\quad
First, we derive the evolution of the negative Sobolev norms of the solution of system \eqref{fullcns0}. 
Similar to the estimates in Section \ref{3}, we need to restrict $s\in(0,3/2)$ to estimate the nonlinear terms.

\begin{lemm}
For $s\in(0,1/2]$, we have
\begin{equation}\label{p7}
\begin{aligned}
\frac{d}{dt}&\int (|\G a^2|+|\G u|^2+|\G\theta|^2)dx+C(\|\nabla\G u\|_{L^2}^2+\|\nabla\G\theta\|_{L^2}^2)\\&\lesssim (\|\nabla a\|_{H^1}^2+\|\nabla\theta\|_{H^1}^2+\|\nabla u\|_{H^2}^2)(\|\G a\|_{L^2}+\|\G u\|_{L^2}+\|\G\theta\|_{L^2}),
\end{aligned}
\end{equation}
and for $s\in(1/2,3/2)$, we have
\begin{equation}\label{p8}
\begin{aligned}
&\frac{d}{dt}\int (|\G a^2|+|\G u|^2+|\G\theta|^2)dx+C(\|\nabla\G u\|_{L^2}^2+\|\nabla\G\theta\|_{L^2}^2)\\&
\quad\lesssim \|(a,u,\theta)\|_{L^2}^{s-1/2}\|\nabla(a,u,\theta)\|_{L^2}^{3/2-s}(\|\nabla(u,\theta)\|_{H^1}+\|\nabla a\|_{L^2})(\|\G a\|_{L^2}+\|\G u\|_{L^2}+\|\G\theta\|_{L^2})\\&
\qquad +\|\nabla u\|_{L^2}^{s+1/2}\|\nabla^2u\|_{L^2}^{3/2-s}\|\G\theta\|_{L^2}.
\end{aligned}
\end{equation}
\end{lemm}
\begin{proof}
Applying $\G$ to $\eqref{fullcns}_1$, $\eqref{fullcns}_2$, $\eqref{fullcns}_3$ and multiplying the resulting by $\G a$, $\G u$, $\G\theta$ respectively, summing up and integrating over $\mathbb R^3$ by parts, we obtain
\begin{equation}\label{p6}
\begin{aligned}
&\frac12\frac{d}{dt}\int(|\G a|^2+|\G u|^2+|\G\theta|^2)dx+\mu\int|\nabla\G u|^2dx+(\mu+\lambda)\int|\dive\G u|^2dx+\int|\nabla\G\theta|^2dx\\&
\quad=\int\G(-a\dive u-u\cdot\nabla a)\cdot\G adx+\int\G(-u\cdot\nabla u)\cdot\G udx\\&
\qquad -\int\G(h(a)[\mu\Delta u+(\mu+\lambda)\nabla\dive u])\cdot\G udx+\int\G(h(a)(\nabla a+\nabla\theta))\cdot\G u\\&
\qquad -\int\G(g(a)\nabla(a\theta))\cdot\G udx-\int\G(\dive(\theta u))\cdot\G\theta dx\\&
\qquad +\int\G(g(a)(2\mu|Du|^2+\lambda(\dive u)^2))\cdot\G\theta dx-\int\G(h(a)\Delta\theta)\cdot\G\theta dx
\overset{\Delta}{=}\sum_{i=1}^8 J_i.
\end{aligned}
\end{equation}
Similar to the estimates in Lemma \ref{l3}, when $s\in(0,1/2]$, 
using \eqref{a.3}, Sobolev interpolation inequality, H\"{o}lder inequality and Young inequality, we obtain
\begin{equation*}
\begin{aligned}
&J_1\lesssim (\|\nabla a\|_{H^1}^2+\|\nabla u\|_{H^1}^2)\|\G a\|_{L^2},\\&
J_2\lesssim \|\nabla u\|_{H^1}^2\|\G u\|_{L^2},\\&
J_3\lesssim (\|\nabla a\|_{H^1}^2+\|\nabla^2u\|_{L^2}^2)\|\G u\|_{L^2},\\&
J_4\lesssim (\|\nabla a\|_{H^1}^2+\|\nabla\theta\|_{L^2}^2)\|\G u\|_{L^2},\\&
J_5\lesssim (\|\nabla\theta\|_{H^1}^2+\|\nabla a\|_{H^1}^2)\|\G u\|_{L^2},\\&
J_6\lesssim (\|\nabla \theta\|_{H^1}^2+\|\nabla u\|_{H^1}^2)\|\G\theta\|_{L^2},\\&
J_7\lesssim \|\nabla u\|_{H^2}^2\|\G\theta\|_{L^2},\\&
J_8\lesssim (\|\nabla a\|_{H^1}^2+\|\nabla^2\theta\|_{L^2}^2)\|\G\theta\|_{L^2}.
\end{aligned}
\end{equation*}
Substituting above estimates from $J_1$ to $J_8$ into \eqref{p6}, we obtain the estimate \eqref{p7}.

Next, if $s\in(1/2,3/2)$, we should deal with $J_1\sim J_8$ in a different way, 
using the different Sobolev interpolation, it is easy to check that
\begin{equation*}
\begin{aligned}
&J_1\lesssim (\|a\|_{L^{3/s}}\|\nabla u\|_{L^2}+\|u\|_{L^{3/s}}\|\nabla a\|_{L^2})\|\G a\|_{L^2}\\&
\quad \lesssim(\|a\|_{L^2}^{s-1/2}\|\nabla a\|_{L^2}^{3/2-s}\|\nabla u\|_{L^2}+\|u\|_{L^2}^{s-1/2}\|\nabla u\|_{L^2}^{3/2-s}\|\nabla a\|_{L^2})\|\G a\|_{L^2},\\&
J_2\lesssim \|u\|_{L^2}^{s-1/2}\|\nabla u\|_{L^2}^{3/2-s}\|\nabla u\|_{L^2}\|\G u\|_{L^2},\\&
J_3\lesssim \|a\|_{L^2}^{s-1/2}\|\nabla a\|_{L^2}^{3/2-s}\|\nabla^2 u\|_{L^2}\|\G u\|_{L^2},\\&
J_4\lesssim \|a\|_{L^2}^{s-1/2}\|\nabla a\|_{L^2}^{3/2-s}(\|\nabla a\|_{L^2}+\|\nabla\theta\|_{L^2})\|\G u\|_{L^2},\\&
J_5\lesssim (\|a\|_{L^2}^{s-1/2}\|\nabla a\|_{L^2}^{3/2-s}\|\nabla \theta\|_{L^2}+\|\theta\|_{L^2}^{s-1/2}\|\nabla \theta\|_{L^2}^{3/2-s}\|\nabla a\|_{L^2})\|\G u\|_{L^2},\\&
J_6\lesssim  (\|\theta\|_{L^2}^{s-1/2}\|\nabla \theta\|_{L^2}^{3/2-s}\|\nabla u\|_{L^2}+\|u\|_{L^2}^{s-1/2}\|\nabla u\|_{L^2}^{3/2-s}\|\nabla \theta\|_{L^2})\|\G \theta\|_{L^2},\\&
J_7\lesssim \|\nabla u\|_{L^2}^{s-1/2}\|\nabla^2 u\|_{L^2}^{3/2-s}\|\nabla u\|_{L^2}\|\G\theta\|_{L^2},\\&
J_8\lesssim \|a\|_{L^2}^{s-1/2}\|\nabla a\|_{L^2}^{3/2-s}\|\nabla^2 \theta\|_{L^2}\|\G\theta\|_{L^2}.
\end{aligned}
\end{equation*}
Plugging the estimates from $J_1$ to $J_8$ into \eqref{p6}, we deduce \eqref{p8}. 
Thus we complete the proof of this lemma.
\end{proof}

Next, let us prove the decay result when the initial data $(a_0,u_0,\theta_0)\in\dot H^{-s}$, $s\in (0,3/2)$. 
Similar to the proof in section \ref{3}, we split the proof into two parts. Before that, we need to introduce the following proposition first, which comes from \cite{he-huang-wang2}.
\begin{prop}\label{prop2}
Under the assumptions of Theorem \ref{0them}, there holds
\begin{equation}\label{prop2est}
\frac{d}{dt}X_2(t)+C\left(\|\nabla^2u\|_{L^2}^2+\|\nabla u\|_{L^2}^2+\|\nabla a\|_{L^2}^2+\|\nabla\dot u\|_{L^2}^2+\|\nabla\theta\|_{L^2}^2+\|\nabla^2\theta\|_{L^2}^2\right)\le 0,
\end{equation}
where $$X_2(t)\sim \|u\|_{H^1}^2+\|a\|_{H^1}^2+\|\dot u\|_{L^2}^2+\|\theta\|_{H^1}^2.$$
\end{prop}

With the Proposition \ref{prop2} at hand, we obtain the following decay result.

\begin{prop}\label{prop3}
Under the sssumptions of Theorem \ref{1theo}, when $s\in (0,1/2]$, then for all $t\ge 0$, we have
\begin{equation*}
\|\G a(t)\|_{L^2}^2+ \|\G u(t)\|_{L^2}^2+\|\G\theta(t)\|_{L^2}^2+\|\G \dot u(t)\|_{L^2}^2\le C,
\end{equation*}
\begin{equation}\label{prop30}
\|a(t)\|_{H^1}^2+\|u(t)\|_{H^1}^2+\|\theta(t)\|_{H^1}^2+\|\dot u(t)\|_{L^2}^2\le C(1+t)^{-s},
\end{equation}
where $C$ is a constant independent of time.
\end{prop}
\begin{proof}
Integrating \eqref{p7} over $[0,t]$, then using the uniform bound \eqref{uniform}, we have
\begin{equation}
\begin{aligned}
&\sup_{\tau\in[0,t]}(\|\G a(\tau)\|_{L^2}^2+\|\G u(\tau)\|_{L^2}^2+\|\G \theta(\tau)\|_{L^2}^2)\\&
\quad \le (\|\G a_0\|_{L^2}^2+\|\G u_0\|_{L^2}^2+\|\G \theta_0\|_{L^2}^2)\\&
\qquad +C\int_0^t(\|\nabla (a,\theta)(\tau)\|_{H^1}^2+\|\nabla u(\tau)\|_{H^2}^2)(\|\G a(\tau)\|_{L^2}+\|\G u(\tau)\|_{L^2}+\|\G\theta(\tau)\|_{L^2})d\tau\\&
\quad \le C+C\sup_{\tau\in[0,t]}(\|\G a(\tau)\|_{L^2}+\|\G u(\tau)\|_{L^2}+\|\G\theta(\tau)\|_{L^2}),
\end{aligned}
\end{equation}
which implies
\begin{equation*}
\|\G a(t)\|_{L^2}^2+\|\G u(t)\|_{L^2}^2+\|\G \theta(t)\|_{L^2}^2\le C \quad {\rm for}~s\in(0,1/2].
\end{equation*}
Similar as \eqref{prop1est1} in Proposition \ref{prop1}, we have
\begin{equation}\label{prop3est1}
\Delta u\in \dot H^{-s}.
\end{equation}
Recall that $a\in \dot H^{-s}\cap H^2$, $\theta\in\dot H^{-s}\cap H^1$, 
we get $a\in \dot H^{1-s}$ and $\theta\in\dot H^{1-s}$, and hence
\begin{equation*}
\begin{aligned}
\||\xi|^{-s}\widehat{\nabla(a\theta)}\|_{L^2}&\lesssim \||\xi|^{1-s}\widehat {a\theta}\|_{L^2}\lesssim \||\xi|^{1-s}\hat a*\hat\theta\|_{L^2}\lesssim \||\xi|^{1-s}\hat a\|_{L^2}\|\hat\theta\|_{L^1}
\lesssim\||\xi|^{1-s}\hat a\|_{L^2}\|\theta\|_{L^\infty}
\le C,
\end{aligned}
\end{equation*}
where we used the condition \eqref{thetarho} in the last inequality. Hence, it is easy to check that
\begin{equation}\label{prop3est2}
\nabla P=\nabla(a\theta)+\nabla a+\nabla\theta\in \dot H^{-s}.
\end{equation}
Since $\rho\dot u+\nabla P=\mu\Delta u+(\mu+\lambda)\nabla\dive u$, combining \eqref{prop3est1} and \eqref{prop3est2}, we have $\rho\dot u\in \dot H^{-s}$. Using $\underline\rho\le \rho\le M$, we deduce that
$$\dot u\in \dot H^{-s}.$$
So far, we have
\begin{equation*}
\|\G a(t)\|_{L^2}^2+\|\G u(t)\|_{L^2}^2+\|\G \theta(t)\|_{L^2}^2+\|\G\dot u(t)\|_{L^2}^2\le C \quad for \quad s\in(0,1/2].
\end{equation*}
Thanks to \eqref{a.2}, it holds on
\begin{equation*}
\|\nabla^{l+1}f\|_{L^2}\ge C\|\G f\|_{L^2}^{-\frac{1}{l+s}}\|\nabla^lf\|_{L^2}^{1+\frac{1}{l+s}}.
\end{equation*}
By this fact and the uniform bound \eqref{uniform}, we deduce that
\begin{equation*}
\|\nabla u\|_{H^1}^2+\|\nabla a\|_{L^2}^2+\|\nabla\dot u\|_{L^2}^2+\|\nabla\theta\|_{H^1}^2\ge C(\|u\|_{H^1}^2+\|a\|_{H^1}^2+\|\dot u\|_{L^2}^2+\|\theta\|_{H^1}^2)^{1+\frac1s}.
\end{equation*}
Substituting above inequality into \eqref{prop2est}, we get
\begin{equation*}
\frac{d}{dt}X_2(t)+CX_2(t)^{1+\frac1s}\le 0.
\end{equation*}
Solving this ODE directly, and using the uniform bound \eqref{uniform}, we have
\begin{equation*}
X_2(t)\le C(1+t)^{-s}, \quad s\in(0,1/2].
\end{equation*}
Therefore we complete the proof of this lemma.
\end{proof}

Similar to the analysis in section \ref{3}, before we derive the propagation of the negative Sobolev norms of the solution for $s\in(1/2,3/2)$, it is important to improve the decay estimate of the first order spatial derivative of the solution for $s\in(0,1/2]$. 

\begin{prop}\label{prop4} 
Under the assumptions of Theorem \ref{1theo}, and $s\in(0,1/2]$, there exists a large time $T_4$ such that
\begin{equation*}
\|\nabla u(t)\|_{H^1}^2+\|\nabla a(t)\|_{H^1}^2+\|\nabla\theta(t)\|_{H^1}^2\le C(1+t)^{-(1+s)}
\end{equation*}
holds on for all $t\ge T_4$,
where $C$ is a constant independent of time.
\end{prop}
\begin{proof}
Thanks to \eqref{a.2}, we have
\begin{equation*}
\|\nabla^{l+1}f\|_{L^2}\ge C\|\G f\|_{L^2}^{-\frac{1}{l+s}}\|\nabla^lf\|_{L^2}^{1+\frac{1}{l+s}},
\end{equation*}
this together with the uniform bound \eqref{uniform}, we obtain
\begin{equation}\label{ooo1}
\|\nabla^2u\|_{H^1}^2+\|\nabla^2\theta\|_{H^1}^2+\|\nabla^2a\|_{L^2}^2\ge C(\|\nabla u\|^2_{H^1}+\|\nabla\theta\|_{H^1}^2+\|\nabla a\|_{H^1}^2)^{1+\frac{1}{1+s}}.
\end{equation}
On the other hand, similar as the proof of Lemma \ref{diss}, the combination of \eqref{pp11} in Lemma \ref{l5} and the decay rate \eqref{prop30} in Proposition \ref{prop3} easily derives that there exits a large time $T_4$ such that
\begin{equation*}
\frac{d}{dt}\left(\|\nabla a\|_{H^1}^2+\|\nabla u\|_{H^1}^2+\|\nabla \theta\|_{H^1}^2+2\delta\int\nabla u\cdot\nabla^2adx\right)
+\mu\|\nabla^2u\|_{H^1}^2+\|\nabla^2\theta\|_{H^1}^2+\delta\|\nabla^2a\|_{L^2}^2\le 0
\end{equation*}
holds on for all $t\ge T_4$.
Thus, by virtue of the equivalent relation \eqref{b4} and \eqref{ooo1}, there holds for all $t\ge T_4$
\begin{equation}\label{ee2} 
\frac{d}{dt}\mathcal E_2^2(t)+C(\mathcal E_2^2(t))^{1+\frac{1}{1+s}}\le 0.
\end{equation}
Solving the inequality \eqref{ee2} directly, and according to \eqref{uniform} and $\|\nabla^2\theta\|_{L^\infty_tL^2}< \infty$, we obtain
\begin{equation*}
\mathcal E_2^2(t)\le C(1+t)^{-(1+s)}, \quad t\ge T_4,
\end{equation*}
which implies
\begin{equation*}
\|\nabla a(t)\|_{H^1}^2+\|\nabla u(t)\|_{H^1}^2+\|\nabla\theta(t)\|_{H^1}^2\le C(1+t)^{-(1+s)}, \quad t\ge T_4.
\end{equation*}
Therefore we complete the proof of this lemma.
\end{proof}
\begin{prop}\label{prop7}  
Under the assumptions of Theorem \ref{1theo}, when $s\in(1/2,3/2)$, then for all $t\ge 0$, we have
\begin{equation*}
\|\G a(t)\|_{L^2}^2+\|\G u(t)\|_{L^2}^2+\|\G \theta(t)\|_{L^2}^2+\|\G\dot u(t)\|_{L^2}^2\le C,\end{equation*}
and
\begin{equation*}
\|a(t)\|_{H^1}^2+\|u(t)\|_{H^1}^2+\|\theta(t)\|_{H^1}^2+\|\dot u(t)\|_{L^2}^2\le C(1+t)^{-s}.
\end{equation*}
For $t\ge T_4$, there holds
\begin{equation*}
\|\nabla a(t)\|_{H^1}^2+\|\nabla u(t)\|_{H^1}^2+\|\nabla\theta(t)\|_{H^1}^2\le C(1+t)^{-(1+s)}, \end{equation*}
where $T_4$ is given in Proposition \ref{prop4}, and $C$ is a constant independent of time.
\end{prop}
\begin{proof}
Notice that $(a_0,u_0,\theta_0)\in \dot H^{-s}\cap L^2\subset \dot H^{-1/2}$ when $s\in(1/2,3/2)$, then we derive from what we have provn in Propositions \ref{prop3} and \ref{prop4} with $s=1/2$ that
\begin{equation}\label{p88}
\begin{aligned}
&\|a(t)\|_{L^2}+\|u(t)\|_{L^2}+\|\theta(t)\|_{L^2}\le C(1+t)^{-\frac14}, \quad t\ge0,\\&
\|\nabla a(t)\|_{H^1}+\|\nabla u(t)\|_{H^1}+\|\nabla \theta(t)\|_{H^1}\le C(1+t)^{-\frac34},\quad t\ge T_4.
\end{aligned}
\end{equation}
Substituting \eqref{p88} into \eqref{p8}, it is easy to obtain that for $s\in(1/2,3/2)$,
\begin{equation*}
\begin{aligned}
&\sup_{\tau\in[0,t]}\left(\|\G a(\tau)\|_{L^2}^2+\|\G u(\tau)\|_{L^2}^2+\|\G \theta(\tau)\|_{L^2}^2\right)+C\int_0^t(\|\nabla\G u\|_{L^2}^2+\|\nabla\G\theta\|_{L^2}^2)d\tau\\&
\quad \lesssim \|\G a_0\|_{L^2}^2+ \|\G u_0\|_{L^2}^2+ \|\G \theta_0\|_{L^2}^2+\int_0^{T_4}(1+\tau)^{-\frac14(s+\frac12)}\|\G\theta\|_{L^2}d\tau\\&
\qquad+\int_0^{T_4}(1+\tau)^{-\frac14}(\|\G a\|_{L^2}+\|\G u\|_{L^2}+\|\G\theta\|_{L^2})d\tau
      +\int_{T_4}^t(1+\tau)^{-\frac32}\|\G\theta\|_{L^2}d\tau\\&
\qquad +\int_{T_4}^t(1+\tau)^{-\frac74+\frac12s}(\|\G a\|_{L^2}+\|\G u\|_{L^2}+\|\G\theta\|_{L^2})d\tau.
\end{aligned}
\end{equation*}
This implies
\begin{equation*}
\|\G a(t)\|_{L^2}+\|\G u(t)\|_{L^2}+\|\G \theta(t)\|_{L^2}\le C \quad {\rm for} ~s\in(1/2,3/2).
\end{equation*}
According to Proposition \ref{prop2}, repeat the progress in Proposition \ref{prop3} for $s\in(1/2,3/2)$, 
there hold  for all $t\ge 0$
\begin{equation*}
\|\G a(t)\|_{L^2}^2+\|\G u(t)\|_{L^2}^2+\|\G \theta(t)\|_{L^2}^2+\|\G\dot u(t)\|_{L^2}^2\le C, \quad s\in(1/2,3/2),
\end{equation*}
and
\begin{equation*}
\|a(t)\|_{H^1}^2+\|u(t)\|_{H^1}^2+\|\theta(t)\|_{H^1}^2+\|\dot u(t)\|_{L^2}^2\le C(1+t)^{-s}, \quad s\in(1/2,3/2).
\end{equation*}
Similarly, repeat the progress in Proposition \ref{prop4} for $s\in(1/2,3/2)$, 
it holds for all $t\ge T_4$
\begin{equation*}
\|\nabla a(t)\|_{H^1}^2+\|\nabla u(t)\|_{H^1}^2+\|\nabla\theta(t)\|_{H^1}^2\le C(1+t)^{-(1+s)}, \quad s\in(1/2,3/2).
\end{equation*}
The combination of above two estimates completes the proof of this lemma.
\end{proof}

\textbf{Proof of Theorem \ref{1theo}:} 
With the uniform bounds and decay estimates stated in Propositions \ref{prop3}, \ref{prop4} and \ref{prop7},
we can complete the the proof of Theorem \ref{1theo}.

\appendix
\section{Appendix. Analytic tools}\label{appendix}

\quad
In this section, we need to introduce some useful lemmas which will be frequently used throughout the paper, the first one is Gagliardo-Nirenberg inequality, the proof can be found in \cite{nirenberg}(pp.125).

\begin{lemm}
Let $0\le m,\alpha\le l$, then we have
\begin{equation}\label{a.1}
\|\nabla^\alpha f\|_{L^p}\lesssim \|\nabla^mf\|_{L^2}^{1-\theta}\|\nabla^lf\|_{L^2}^\theta
\end{equation}
where $0\le\theta\le1$ and $\alpha$ satisfy
\begin{equation*}
\frac1p-\frac\alpha3=(\frac12-\frac m3)(1-\theta)+(\frac12-\frac l3)\theta.
\end{equation*}
\end{lemm}

Next we introduce a special Sobolev interpolation inequality, the proof can be found in \cite{guo-wang}(see Lemma A.4). 
\begin{lemm}\label{l1}
Let $s\ge 0$ and $l\ge 0$, then we have
\begin{equation}\label{a.2}
\|\nabla^l f\|_{L^2}\lesssim \|\nabla ^{l+1}f\|^{1-\alpha}_{L^2}\|\Lambda ^{-s}f\|^{\alpha}_{L^2},
\end{equation}
where $\alpha=\frac{1}{l+1+s}$.
\end{lemm}

If $s\in(0,3)$, $\G f$ is the Riesz potential. According to the Hardy-Littlewood-Sobolev theorem, there is the following inequality, the proof of this inequality see \cite{stein}.
\begin{lemm}\label{l2}
Let $0<s<3$, $1<p<q<\infty$, $\frac1q+\frac s3=\frac1p$, then
\begin{equation}\label{a.3}
\|\Lambda^{-s}f\|_{L^q}\lesssim \|f\|_{L^p}.
\end{equation}
\end{lemm}

The following Hausdorff-Young inequality is useful in this paper. The proof can be found in \cite{loukas}(see Proposition 2.2.16).
\begin{lemm}
When $f\in L^p(\mathbb R^3)$, $1\le p\le 2$, then $\hat f\in L^{p'}(\mathbb R^n)$, and there holds
\begin{equation}\label{hausdorff}
\|\hat f\|_{L^{p'}}\le C\|f\|_{L^p},
\end{equation}
where $1/p+1/p'=1$.
\end{lemm}

\section*{Acknowledgements}

This research was partially supported by NSFC (Grant Nos.11801586, 11971496 and 11431015) and the
Fundamental Research Funds for the Central Universities of China (Grant No.18lgpy66).

\phantomsection

\end{document}